\documentclass[leqno,12pt]{article}
\usepackage[hypertex]{hyperref}%
\usepackage{amsfonts,amsmath,amssymb}
\usepackage{amsthm}
\usepackage{enumerate}
\usepackage{graphics}
\usepackage{epic}
\usepackage{eepic}
\usepackage{float}%
\usepackage{enumerate}
\numberwithin{equation}{section} %
\numberwithin{table}{section}

\newcommand{\ZG}[1]{{\mathfrak {Z}}_{#1}}
\newcommand{\HomCalg}[2]{\operatorname{Hom}_{{\mathbb{C}}\text{-alg}}
({#1},{#2})}
\newcommand{\invHom}[3]{\operatorname{Hom}_{#1}({#2},{#3})}
\newcommand{\ruby}[2]{%
\leavevmode
\setbox0=\hbox{#1}%
\setbox1=\hbox{\tiny #2}%
\ifdim\wd0>\wd1 \dimen0=\wd0 \else \dimen0 =\wd1 \fi
\hbox{%
\kanjiskip=0pt plus 2fil
\xkanjiskip=0pt plus 2fil
\vbox {%
\hbox to \dimen0{%
\tiny \hfil#2\hfil}%
\nointerlineskip
\hbox to \dimen0{\mathstrut\hfil#1\hfil}}}}
  \newtheorem{theorem}{Theorem}[section] %

  \newtheorem{proposition}[theorem]{Proposition} %
  \newtheorem{lemma}[theorem]{Lemma} %
  \newtheorem{corollary}[theorem]{Corollary}%
  \newtheorem{definition}[theorem]{Definition} %
  \newtheorem{example}[theorem]{Example}
  \newtheorem{problem}[theorem]{Problem}
  \newtheorem{remark}[theorem]{Remark}%
\begin{document}
\title{Shintani functions, 
 real spherical manifolds,
 and symmetry breaking operators
%\\
%\small{{\it{In honour of Professor Dynkin 
%for his 90th birthday}}}
}
\author{Toshiyuki KOBAYASHI
\footnote
{Kavli IPMU (WPI) and Graduate School of Mathematical Sciences,
 the University of Tokyo, 
Meguro-ku, Tokyo, 153-8914, Japan, 
E-mail address: toshi@ms.u-tokyo.ac.jp}
}
\date{} %
\maketitle %
\noindent
\begin{abstract}
For a pair of reductive groups $G \supset G'$, 
we prove a geometric criterion
 for the space $\operatorname{Sh}(\lambda, \nu)$
 of Shintani functions 
 to be finite-dimensional
 in the Archimedean case. 
This criterion leads us to a complete classification
 of the symmetric pairs $(G,G')$
 having finite-dimensional Shintani spaces.  
A geometric criterion
 for uniform boundedness
 of $\dim_{\mathbb{C}}{\operatorname{Sh}}(\lambda, \nu)$
 is also obtained.  
Furthermore, 
 we prove that symmetry breaking operators of the restriction
 of smooth admissible representations yield Shintani functions
 of moderate growth,
 of which the dimension is determined
 for $(G, G') = (O(n+1,1), O(n,1))$.
\end{abstract}
\par\noindent
{\bf{MSC 2010;}}\enspace
primary 22E46%   Semisimple Lie groups and their representations
;
secondary 
11F70, %   Representation-theoretic methods; automorphic representations over local and global fields
14M27, %Compactifications; symmetric and spherical varieties
32N15, %Automorphic functions in symmetric domains
53C35. %   Symmetric spaces.  

\tableofcontents

\section{Introduction}
\label{sec:intro}
\setcounter{subsection}{0}
The object of this article 
 is to investigate Shintani functions
 for a pair of reductive groups $G \supset G'$
 in the Archimedean case.  
Among others,
 we classify 
 the reductive symmetric pairs $(G,G')$
 such that the Shintani spaces $\operatorname{Sh}(\lambda,\nu)$
  are finite-dimensional
 for all $(\ZG{G}, \ZG{G'})$-infinitesimal character $(\lambda,\nu)$.  
Explicit dimension formulae
 for the Shintani spaces
 of moderate growth are determined
 for the pair $(G,G') = (O(n+1,1), O(n,1))$.

\vskip 1pc
Let $G$ be a real reductive linear Lie group.  
We write ${\mathfrak {g}}$
 for the Lie algebra of $G$, 
 and $U({\mathfrak {g}}_{\mathbb{C}})$
 for the universal enveloping algebra
 of the complexified Lie algebra 
${\mathfrak {g}}_{\mathbb{C}}:={\mathfrak {g}} \otimes_{\mathbb{R}}
{\mathbb{C}}$.

For $X \in {\mathfrak {g}}$
 and $f \in C^{\infty}(G)$, 
we set
\begin{equation}
\label{eqn:LR}
  (L_X f)(g):=\frac{d}{dt}|_{t=0}f(\exp(-tX)g), 
\quad 
   (R_X f)(g):=\frac{d}{dt}|_{t=0}f(g\exp(tX)), 
\end{equation}
 and extend these actions
 to those of $U({\mathfrak {g}}_{\mathbb{C}})$.

We denote by $\ZG {G}$
 the ${\mathbb{C}}$-algebra of $G$-invariant elements
 in $U({\mathfrak {g}}_{\mathbb{C}})$.  
Let ${\mathfrak{j}}$ be a Cartan subalgebra
 of ${\mathfrak {g}}$.  
Then any $\lambda \in {\mathfrak {j}}_{\mathbb{C}}^{\vee}$
 gives rise to a ${\mathbb{C}}$-algebra homomorphism
 $\chi_{\lambda}:\ZG G \to {\mathbb{C}}$
 via the Harish-Chandra isomorphism
 $\ZG G \overset \sim \to S({\mathfrak {j}}_{\mathbb{C}})
  ^{W({\mathfrak {j}}_{\mathbb{C}})}$, 
 where $W({\mathfrak {j}}_{\mathbb{C}})$ is some finite group
 (see Section \ref{subsec:HC}).

Suppose that $G'$ is an algebraic reductive subgroup.  
Analogous notation will be applied to $G'$.  
For instance, 
 $\HomCalg{\ZG{G'}}{{\mathbb{C}}}
 \simeq ({\mathfrak {j}}_{\mathbb{C}}')^{\vee}/
 W({\mathfrak {j}}_{\mathbb{C}}')$, 
 $\chi_{\nu} \leftrightarrow \nu$, 
 where ${\mathfrak {j}}'$ is a Cartan subalgebra
 of the Lie algebra
 ${\mathfrak {g}}'$ of $G'$.

We take a maximal compact subgroup $K$ of $G$
 such that $K':=K \cap G'$
 is a maximal compact subgroup.  
Following Murase--Sugano \cite{murasesugano}, 
we call:
\begin{definition}
[Shintani function]
\label{def:Sh}
{\rm{
We say $f \in C^{\infty}(G)$ is a 
 {\it{Shintani function}}
 of $(\ZG G, \ZG{G'})$-infinitesimal characters
 $(\lambda,\nu)$
 if $f$ satisfies the following three properties:
\begin{enumerate}
\item[{\rm{(1)}}]
$f(k' g k) =f(g)$
 for any $k' \in K'$, $k \in K$.  
\item[{\rm{(2)}}]
$R_u f = \chi_{\lambda}(u)f$
 for any $u \in \ZG G$.  
\item[{\rm{(3)}}]
$L_v f=\chi_{\nu}(v)f$
for any $v \in \ZG {G'}$.  
\end{enumerate}
}}
\end{definition}

We denote by $\operatorname{Sh}(\lambda,\nu)$
 the space of Shintani functions 
 of type $(\lambda,\nu)$.  

For $G=G'$
 and $\lambda=-\nu$, 
Shintani functions are nothing but Harish-Chandra's zonal spherical 
 functions.  

\vskip 1pc
In this article, 
 we provide
 the following three different realizations
 of the Shintani space $\operatorname{Sh} (\lambda,\nu)$:

$\bullet$ \enspace
Matrix coefficients
 of symmetry breaking operators.  
(See Proposition \ref{prop:break}.)  

$\bullet$ \enspace
$(K \times K')$-invariant functions
 on $(G \times G')/\operatorname{diag}G'$.  
(See Lemma \ref{lem:ShPP}.)

$\bullet$ \enspace
$G'$-invariant functions
 on the Riemannian symmetric space
 $(G \times G')/(K \times K')$.  
(See Lemma \ref{lem:ShGK}.)

\vskip 1pc
The first realization constructs
 Shintani functions having moderate growth
 (Definition \ref{def:mod}) from the restriction 
of admissible smooth representations of $G$
 with respect to the subgroup $G'$, 
whereas the second realization relates 
 $\operatorname{Sh} (\lambda,\nu)$
 with the theory of real spherical homogeneous spaces
 which was studied 
 in \cite{xtoshi95, xkeastwood60, xtoshimatsuki, xtoshitoshima}.  
Via the third realization,
 we can apply powerful methods
({\it{e.g.}},  
 \cite{KKMOOT})
 of harmonic analysis 
 on Riemannian symmetric spaces
 for the study of Shintani functions.

By using these ideas,
 we give a characterization
 of the pair $(G,G')$ 
 for which the Shintani space 
$\operatorname{Sh} (\lambda,\nu)$
 is finite-dimensional
 for all $(\lambda,\nu)$:
\begin{theorem}
[see Theorem \ref{thm:dozen}]
\label{thm:A}
The following four conditions
on a pair of real reductive
 algebraic groups $G\supset G'$
 are equivalent:
\begin{enumerate}
\item[{\rm{(i)}}]
{\rm{(Shintani function)}}\enspace
$\operatorname{Sh}(\lambda,\nu)$
 is finite-dimensional for any  
 pair $(\lambda,\nu)$ of $(\ZG G, \ZG{G'})$-infinitesimal characters.
\item[{\rm{(ii)}}]
{\rm{(Symmetry breaking)}}\enspace
$\invHom{G'}{\pi^{\infty}}{\tau^{\infty}}$
 is finite-dimensional 
 for any pair $({\pi^{\infty}},{\tau^{\infty}})$
 of admissible smooth representations
 of $G$ and $G'$
 {\rm{(}}see {\rm{Section \ref{subsec:adm}}}{\rm{)}}.  
\item[{\rm{(iii)}}]
{\rm{(Invariant bilinear form)}}\enspace
There exist at most finitely many 
 linearly independent
 $G'$-invariant bilinear forms
 on $\pi^{\infty} \otimes \tau^{\infty}$
 for any pair $({\pi^{\infty}},{\tau^{\infty}})$
 of admissible smooth representations
 of $G$ and $G'$.  
\item[{\rm{(iv)}}]
{\rm{(Geometric property (PP))}}\enspace
There exist minimal parabolic subgroups
 $P$ and $P'$
 of $G$ and $G'$, 
respectively,
 such that $PP'$ is open in $G$.  
\end{enumerate}
\end{theorem}
The dimension of the Shintani space $\operatorname{Sh} (\lambda,\nu)$
 depends on $\lambda$ and $\nu$ in general.  
We give a characterization
 of the uniform boundedness property:
\begin{theorem}
\label{thm:B}
The following four conditions
 on a pair of real reductive algebraic groups
 $G \supset G'$
 are equivalent:
\begin{enumerate}
\item[{\rm{(i)}}]
{\rm{(Shintani function)}}\enspace
There exists a constant $C$
 such that 
\[
 \dim_{\mathbb{C}}
 \operatorname{Sh}(\lambda,\nu)
 \le 
 C
\]
 for any pair $(\lambda,\nu)$ of $(\ZG G, \ZG {G'})$-infinitesimal 
 characters. 

\item[{\rm{(ii)}}]
{\rm{(Symmetry breaking)}}\enspace
There exists a constant $C$
 such that
\[
\dim_{\mathbb{C}}
\invHom{G'}{\pi^{\infty}}{\tau^{\infty}}
\le C
\]
 for any pair $({\pi^{\infty}},{\tau^{\infty}})$
 of admissible smooth representations
 of $G$ and $G'$.  
\item[{\rm{(iii)}}]
{\rm{(Invariant bilinear form)}}\enspace
There exists a constant $C$
 such that 
\[
\dim_{\mathbb{C}}
\invHom{G'}{\pi^{\infty} \otimes \tau^{\infty}}
{{\mathbb{C}}}\le C
\]
 for any pair $({\pi^{\infty}},{\tau^{\infty}})$
 of admissible smooth representations
 of $G$ and $G'$.  
\item[{\rm{(iv)}}]
{\rm{(Geometric property (BB))}}\enspace
There exist Borel subgroups
 $B$ and $B'$
 of the complex Lie groups $G_{\mathbb{C}} \supset G_{\mathbb{C}}'$
 with Lie algebras 
 ${\mathfrak {g}}_{\mathbb{C}} \supset {\mathfrak {g}}_{\mathbb{C}}'$, 
 respectively,
 such that $BB'$ is open in $G_{\mathbb{C}}$.  
\end{enumerate}
\end{theorem}
By using the geometric criterion (PP), 
we give a complete classification
 of the reductive symmetric pairs
 $(G,G')$
 for which one of 
 (therefore any of)
the equivalent conditions
 in Theorem \ref{thm:A}
 is fulfilled.  
See Theorem \ref{thm:PP}
 for the classification.  
Among them,
 those satisfying the uniform boundedness property
 in Theorem \ref{thm:B}
 are listed in Theorem \ref{thm:BB}.  
\begin{example}
[see Theorems \ref{thm:PP} and \ref{thm:BB}]
\label{ex:PPBB}
{\rm{
\begin{enumerate}
\item[1)]
If $(G,G')$ is 
\begin{alignat*}{2}
&(GL(n+1,{\mathbb{C}}),GL(n,{\mathbb{C}})\times GL(1,{\mathbb{C}}))
\qquad
&&(n \ge 1),
\\
&(O(n+1,{\mathbb{C}}),O(n,{\mathbb{C}}))
\qquad
&&(n \ge 1),
\end{alignat*}
or any real form
 of them,
 then we have
\begin{equation}
\label{eqn:Shbdd}
  \sup_{\lambda} \sup_{\nu}
  \dim_{{\mathbb{C}}}
  \operatorname{Sh}(\lambda,\nu)
   <\infty.  
\end{equation}
\item[2)]
If $(G,G')$ is 
\begin{equation*}
(Sp(n+1,{\mathbb{C}}),Sp(n,{\mathbb{C}}) \times Sp(1,{\mathbb{C}}))
\qquad
(n \ge 2),
\end{equation*}
or its split real form, 
 then $\operatorname{Sh}(\rho_{\mathfrak {g}},\rho_{\mathfrak {g}'})$
 is infinite-dimensional
 (see \eqref{eqn:rhog} for the notation).  
On the other hand,
 if $(G,G')$ is a non-split real form, 
 then $\operatorname{Sh}(\lambda,\nu)$
 is finite-dimensional 
 for all $(\lambda,\nu)$,
 but the dimension is not uniformly bounded, 
 namely, 
 \eqref{eqn:Shbdd} fails.  

\item[3)]
If $(G,G')$ is 
\[
 (GL(n+1,{\mathbb{H}}), 
 GL(n,{\mathbb{H}}) \times GL(1,{\mathbb{H}}))
\quad
 (n \ge 1), 
\]
 then $\operatorname{Sh} (\lambda,\nu)$
 is finite-dimensional for all $(\lambda,\nu)$, 
 but \eqref{eqn:Shbdd} fails.  
\end{enumerate}
}}
\end{example}

This article is organized as follows:

In Section \ref{sec:PP}, 
 we give a complete list
 of the reductive symmetric pairs $(G,G')$
 such that the dimension of the Shintani space
 is finite/uniformly bounded.

After a brief review 
 on basic results
 on continuous
 (infinite-dimensional) representations
 of real reductive Lie groups
 in Section \ref{sec:pre}, 
 we enrich Theorem \ref{thm:A}
 by adding some more conditions that are equivalent
 to the finiteness of 
 $\dim_{\mathbb{C}} \operatorname{Sh} (\lambda,\nu)$
 in Theorem \ref{thm:dozen}.

The upper estimate of 
 $\dim_{\mathbb{C}} \operatorname{Sh} (\lambda,\nu)$
 is proved in Section \ref{sec:sph}
 by using the theory
 of {\it{real spherical}} homogeneous spaces
 which was established in \cite{xtoshitoshima}.

In Section \ref{sec:break}
 we give a lower estimate 
 of $\dim_{\mathbb{C}} \operatorname{Sh} (\lambda,\nu)$
 by using the intertwining operators 
 constructed in Section \ref{sec:Poisson}.

In Section \ref{sec:PS}
 we apply the theory
 of harmonic analysis
 on Riemannian symmetric spaces,
 and investigate the relationship
 between symmetry breaking operators
 of the restriction 
 of admissible smooth representations
 of $G$ to $G'$
 and Shintani functions.  
Section \ref{sec:sbon}
 provides an example
 for $(G,G')=(O(n+1,1),O(n,1))$
 by using a recent work
 \cite{xtsbon}
 with B. Speh
 on symmetry breaking operators.

\section{Classification of 
 $(G,G')$ 
 with 
$\dim_{\mathbb{C}}
 \operatorname{Sh}(\lambda,\nu)
 <\infty$
}
\label{sec:PP}
This section gives a complete
 classification
 of the reductive symmetric pairs
 $(G,G')$
 such that 
 the dimension of the Shintani space
 $\operatorname{Sh}(\lambda,\nu)$
 is finite/bounded
 for any $(\ZG G,\ZG {G'})$-infinitesimal 
 characters $(\lambda,\nu)$.  
Owing to the criteria in 
 Theorems \ref{thm:A} and \ref{thm:B}, 
 the classification is reduced to that 
 of (real) spherical homogeneous spaces
 of the form $(G \times G')/\operatorname{diag}G'$, 
 which was accomplished in \cite{xtoshimatsuki}.  
\begin{definition}
[Symmetric pair]
\label{def:symm}
{\rm{
Let $G$ be a real reductive Lie group.  
We say $(G,G')$ is a 
 {\it{reductive symmetric pair}}
 if $G'$ is an open subgroup 
 of the fixed point subgroup 
 $G^{\sigma}$
 of some involutive automorphism $\sigma$
 of $G$.  
}}
\end{definition}
\begin{example}
\label{ex:symm}
{\rm{
{\rm{1)}}\enspace
 (Group case)\enspace
Let $G_1$ be a Lie group.  
Then the pair 
\[
   (G,G')=(G_1 \times G_1, \operatorname{diag}G_1)
\]
forms a symmetric pair
 with the involution
 $\sigma \in \operatorname{Aut}(G)$
 defined by 
 $\sigma(x,y)=(y,x)$.  
Since the homogeneous space $G/G'$ is isomorphic
 to the group manifold $G_1$
 with $(G_1 \times G_1)$-action from 
the left and the right,
 the pair $(G_1 \times G_1, \operatorname{diag}G_1)$
 is sometimes referred to
 as the group case.  
\par\noindent
{\rm{2)}}\enspace
(Riemannian symmetric pair)\enspace
Let $K$ be a maximal compact subgroup 
of a real reductive linear Lie group $G$.  
Then the pair $(G,K)$ is a symmetric pair
 because $K$ is the fixed point subgroup
 of a Cartan involution $\theta$ of $G$.  
Since the homogeneous space $G/K$ becomes a symmetric space
 with respect to the Levi-Civita connection of a $G$-invariant Riemannian metric
 on $G/K$, 
the pair $(G,K)$ is sometimes referred
 to as a Riemannian symmetric pair.  
}}
\end{example}

The classification of reductive symmetric pairs
 was established by Berger \cite{Berger}
 on the level of Lie algebras.  
Among them we list the pairs $(G,G')$ 
 such that the space of Shintani functions
 is finite-dimensional as follows:

\begin{theorem}
\label{thm:PP}
Suppose $(G,G')$ is a reductive symmetric pair.  
Then the following two conditions
 are equivalent:
\begin{enumerate}
\item[{\rm{(i)}}]
$\operatorname{Sh}(\lambda,\nu)$
 is finite-dimensional 
 for any $(\ZG G, \ZG {G'})$-infinitesimal 
 characters $(\lambda,\nu)$.  
\item[{\rm{(ii)}}]
The pair $({\mathfrak{g}},{\mathfrak{g}}')$ of the Lie algebras 
 is isomorphic 
 {\rm{(}}up to outer automorphisms{\rm{)}}
 to the direct sum of the following pairs:
\begin{enumerate}
\item[{\rm{A)}}]
{\rm{Trivial case:}}
${\mathfrak{g}}={\mathfrak{g}}'$.  
\item[{\rm{B)}}]
{\rm{Abelian case:}}
${\mathfrak {g}}={\mathbb{R}}$, 
${\mathfrak {g}}'=\{0\}$.  
\item[{\rm{C)}}]
{\rm{Compact case:}}
${\mathfrak {g}}$ is the Lie algebra
 of a compact simple Lie group.  
\item[{\rm{D)}}]
{\rm{Riemannian symmetric pair:}}
${\mathfrak {g}}'$ is the Lie algebra 
 of a maximal compact subgroup $K$
 of a non-compact simple Lie group $G$.  
\item[{\rm{E)}}]
{\rm{Split rank one case ($\operatorname{rank}_{{\mathbb{R}}}G=1$):}}
\begin{enumerate}
\item[{\rm{E1)}}]
$({\mathfrak{o}}(p+q,1),
{\mathfrak{o}}(p)+{\mathfrak{o}}(q,1))$
\quad\quad\,\, $(p+q \ge 2)$.
\item[{\rm{E2)}}]
$({\mathfrak{su}}(p+q,1),
{\mathfrak{s}}({\mathfrak {u}}(p)+{\mathfrak{u}}(q,1)))$
\quad
$(p+q \ge 1)$.
\item[{\rm{E3)}}]
$({\mathfrak{sp}}(p+q,1),
{\mathfrak{sp}}(p)+{\mathfrak{sp}}(q,1))$
\quad\,
$(p+q \ge 1)$.
\item[{\rm{E4)}}]
$({\mathfrak{f}}_{4(-20)},
{\mathfrak{o}}(8,1))$.  
\end{enumerate}
\item[{\rm{F)}}]
{\rm{Strong Gelfand pairs
 and their real forms:}}
\begin{enumerate}
\item[{\rm{F1)}}]
$({\mathfrak{sl}}(n+1,{\mathbb{C}}),
{\mathfrak{gl}}(n,{\mathbb{C}}))$
\, $(n\ge 2)$.  
\item[{\rm{F2)}}]
$({\mathfrak{o}}(n+1,{\mathbb{C}}),
{\mathfrak{o}}(n,{\mathbb{C}}))$
\,\,\,\, $(n\ge 2)$.  
\item[{\rm{F3)}}]
$({\mathfrak{sl}}(n+1,{\mathbb{R}}),
{\mathfrak{gl}}(n,{\mathbb{R}}))$ 
\,\,\,$(n\ge 1)$.  
\item[{\rm{F4)}}]
$({\mathfrak{su}}(p+1,q),{\mathfrak{u}}(p,q))$
\,\,\,\,\, $(p+q\ge 1)$.  
\item[{\rm{F5)}}]
$({\mathfrak{o}}(p+1,q),{\mathfrak{o}}(p,q))$ 
\,\,\,\,\,\,\,\,\,\,$(p+q\ge 2)$.  
\end{enumerate}
\item[{\rm{G)}}]
$({\mathfrak{g}}, {\mathfrak{g}}')=
 ({\mathfrak{g}}_1+{\mathfrak{g}}_1, \operatorname{diag} {\mathfrak{g}}_1)$
{\rm{Group case:}}
\begin{enumerate}
\item[{\rm{G1)}}]
${\mathfrak{g}}_1$ is the Lie algebra
 of a compact simple Lie group.  
\item[{\rm{G2)}}]
$({\mathfrak{o}}(n,1)+{\mathfrak{o}}(n,1), \operatorname{diag}
{\mathfrak{o}}(n,1))$
\quad
$(n \ge 2)$.  
\end{enumerate}
\item[{\rm{H)}}]
{\rm{Other cases:}}
\begin{enumerate}
\item[{\rm{H1)}}]
$({\mathfrak{o}}(2n, 2),
{\mathfrak{u}}(n,1))$
\hphantom{MMMMMMMMM}
$(n \ge 1)$.
\item[{\rm{H2)}}]
$({\mathfrak{su}}^{\ast}(2n+2),
{\mathfrak{su}}(2)+{\mathfrak{su}}^{\ast}(2n)+\mathbb{R})$ 
\quad
$(n\ge 1)$.  
\item[{\rm{H3)}}]
$({\mathfrak{o}}^{\ast}(2n+2),
{\mathfrak{o}}(2)+{\mathfrak{o}}^{\ast}(2n))$
\hphantom{MMMMn} $(n\ge 1)$.
\item[{\rm{H4)}}]
$({\mathfrak{sp}}(p+1,q),
{\mathfrak{sp}}(p,q)+{\mathfrak{sp}}(1))$.  
\item[{\rm{H5)}}]
$({\mathfrak{e}}_{6(-26)},
{\mathfrak{so}}(9,1)+{\mathbb{R}})$.
\end{enumerate}
\end{enumerate}

\end{enumerate}
\end{theorem}

We single out 
 those pairs $(G,G')$
 having the uniform boundedness property as follows:
\begin{theorem}
\label{thm:BB}
Suppose $(G,G')$ is a reductive symmetric pair.  
Then the following conditions are equivalent:

\begin{enumerate}
\item[{\rm{(i)}}]
There exists a constant
 such that 
\[
  \dim_{\mathbb{C}}
 \operatorname{Sh}(\lambda,\nu)
 \le C
\]
 for any $(\ZG G, \ZG {G'})$-infinitesimal 
 characters $(\lambda,\nu)$.  
\item[{\rm{(ii)}}]
The pair of the Lie algebras 
 $({\mathfrak{g}},{\mathfrak{g}}')$
 is isomorphic 
 {\rm{(}}up to outer automorphisms{\rm{)}}
 to the direct sum of the pairs in {\rm{(A)}}, {\rm{(B)}}
 and {\rm{(F1)}} -- {\rm{(F5)}}.  
\end{enumerate}
\end{theorem}
\begin{example}
\label{ex:AA}
{\rm{
In connection with branching problems,
 some of the pairs appeared earlier
 in the literatures.  
For instance,
\begin{enumerate}
\item[1)]
{\rm{(Strong Gelfand pairs
\cite{Kr})}}\enspace
(F1), (F2).  

\item[2)]
{\rm{(the Gross--Prasad conjecture \cite{GrossPrasad})}}\enspace
(F2), (F5).  
\item[3)]
{\rm{(Finite-multiplicity for tensor products
 \cite{xtoshi95})}}\enspace
(G2).  
\item[4)]
{\rm{(Multiplicity-free restriction
\cite{aizenbudgourevitch, SunZhu})}}\enspace
(F1)--(F5).  
\end{enumerate}
}}
\end{example}
\begin{remark}
\label{rem:BB}
{\rm{
The following pairs $(G,G')$ are non-symmetric pairs
such that $(G,G')$ satisfies 
 the condition (i) of Theorem \ref{thm:BB}.  
\begin{equation*}
(G,G')=(SO(8,{\mathbb{C}}),Spin(7,{\mathbb{C}})),
\quad
(SO(4,4),Spin(4,3)).  
\end{equation*}
In fact 
 the Lie algebras $({\mathfrak {g}}, {\mathfrak {g}}')$
 are symmetric pairs, 
 but the involution 
 of ${\mathfrak{g}}$
 does not lift to the group $G$.  
}}
\end{remark}

\begin{proof}
[Proof of Theorem \ref{thm:PP}]
Direct from Theorem \ref{thm:A}
 and \cite[Theorem 1.3]{xtoshimatsuki}.  
\end{proof}
\begin{proof}
[Proof of Theorem \ref{thm:BB}]
Direct from Theorem \ref{thm:B}
 and \cite[Proposition 1.6]{xtoshimatsuki}.  
\end{proof}

\section{Preliminary results}
\label{sec:pre}
We begin with a quick review of some basic results
 on (infinite-dimensional) continuous representations
 of real reductive Lie groups.  

\subsection{Continuous representations
 and the Frobenius reciprocity}
\label{subsec:rep}

By a continuous representation $\pi$
 of a Lie group $G$
 on a topological vector space $V$
 we shall mean 
 that $\pi:G \to GL_{\mathbb{C}}(V)$ is a group homomorphism
 such that the induced map 
 $G \times V \to V$, 
 $(g,v) \mapsto \pi(g)v$
 is continuous.  
We say $\pi$ is a (continuous) Hilbert 
 [Banach, Fr{\'e}chet, $\cdots$]
 representation
 if $V$ is a Hilbert [Banach, Fr{\'e}chet, $\cdots$]
space.  
We note
 that a continuous Hilbert representation 
 is not necessarily a unitary representation;
 a Hilbert representation $\pi$ of $G$
 is said to be a unitary representation
 provided
 that all the operators $\pi(g)$
 ($g \in G$)
 are unitary.

Suppose $\pi$ is a continuous representation 
 of $G$ on a Banach space $V$.  
A vector $v \in V$ is said
 to be {\it{smooth}}
 if the map
 $G \to V$, 
 $g \mapsto \pi(g)v$ is of $C^{\infty}$-class.  
Let $V^{\infty}$ denote
 the space
 of smooth vectors
 of the representation $(\pi,V)$.  
Then $V^{\infty}$ carries 
 a Fr{\'e}chet topology
 with a family of semi-norms
$\|v\|_{i_1\cdots i_k}:=\|d\pi(X_{i_1}) \cdots d\pi(X_{i_k})v\|$, 
 where $\{X_1, \cdots, X_n\}$ is a basis
 of ${\mathfrak {g}}$.  
Then $V^{\infty}$ is a $G$-invariant subspace
 of $V$, 
 and we obtain a continuous Fr{\'e}chet representation 
 $(\pi^{\infty}, V^{\infty})$
 of $G$.  

Suppose that $G'$ is another Lie group.  
If $\pi$ and $\tau$ are Hilbert representations
 of $G$ and $G'$
 on the Hilbert spaces ${\mathcal{H}}_{\pi}$
 and ${\mathcal{H}}_{\tau}$, 
 respectively,
 then we can define a continuous Hilbert representation 
 $\pi \boxtimes \tau$
 of the direct product group on the Hilbert completion 
 on ${\mathcal{H}}_{\pi} \widehat\otimes {\mathcal{H}}_{\tau}$
 of the pre-Hilbert space 
 ${\mathcal{H}}_{\pi} \otimes {\mathcal{H}}_{\tau}$.

Suppose further 
 that $G'$ is a subgroup of $G$.  
Then we may regard 
 $\pi$ as a representation of $G'$
 by the restriction.  
The resulting representation is denoted 
 by $\pi|_{G'}$.  
The restriction
 of the outer tensor product $\pi \boxtimes \tau$
 of $G \times G'$
 to the subgroup 
 $\operatorname{diag} G'
  =\{(g', g'): g' \in G'\}$
 is denoted by $\pi \otimes \tau$.  
By a {\it{symmetry breaking operator}}
 we mean a continuous $G'$-homomorphism from the 
 representation space of $\pi$
 to that of $\tau$.  
We write $\invHom{G'}{\pi|_{G'}}{\tau}$
 for the vector space
 of continuous $G'$-homomorphisms.  
Analogous notation
 is applied to smooth representations.

For the convenience
 of the reader, 
we review some basic properties
 of the restriction:
\begin{lemma}
\label{lem:rest}
Suppose that $\pi$ and $\tau$
 are Hilbert representations
 of $G$ and $G'$
 on Hilbert spaces
 ${\mathcal{H}}_{\pi}$ and ${\mathcal{H}}_{\tau}$, 
respectively.  
\begin{enumerate}
\item[{\rm{1)}}]
 There is a canonical injective homomorphism:
\begin{equation}
\invHom {G'}{\pi|_{G'}}{\tau}
 \hookrightarrow 
\invHom{G'}{\pi^{\infty}|_{G'}}{\tau^{\infty}}, 
\qquad
T \mapsto T|_{{\mathcal{H}}_{\pi}^{\infty}}.  
\label{eqn:Hsmooth}
\end{equation}
\item[{\rm{2)}}]
Let $\tau^{\vee}$ be the contragredient representation 
 of $\tau$.  
Then we have a canonical isomorphism:
\begin{equation}
\invHom {G'}{\pi|_{G'}}{\tau} 
\simeq 
\invHom{G'}{\pi \otimes \tau^{\vee}}{{\mathbb{C}}}.  
\label{eqn:Hrest}
\end{equation}
\item[{\rm{3)}}]
There is a canonical injective homomorphism
 if $G$ and $G'$ are real reductive:
\[
\invHom{G'}{\pi^{\infty}|_{G'}}{\tau^{\infty}}
\hookrightarrow
\invHom{G'}{\pi^{\infty}\otimes (\tau^{\vee})^{\infty}}{\mathbb{C}}.  
\]
\end{enumerate}
\end{lemma}
\begin{proof}
1)\enspace 
See \cite[Lemma 5.1]{xtoshitoshima},
 for instance.  
2) \enspace
We have a canonical isomorphism between 
$\operatorname{Hom}_{\mathbb{C}}
 ({\mathcal{H}}_{\pi}, {\mathcal{H}}_{\tau})$
 and 
$\operatorname{Hom}_{\mathbb{C}}
 ({\mathcal{H}}_{\pi} \widehat \otimes {\mathcal{H}}_{\tau}^{\vee}, {\mathbb{C}})$, 
 where $\operatorname{Hom}_{\mathbb{C}}(\,\, ,\,\,)$
 denotes the space of continuous linear maps.  
Taking $G'$-invariant elements,
 we get \eqref{eqn:Hrest}.  
3) \enspace
See \cite[Lemma A.0.8]{aizenbudgourevitch}, 
 for instance.  
\end{proof}

\begin{proposition}
[Frobenius reciprocity]
\label{prop:Frob}
Let $H$ be a closed subgroup
of a Lie group $G$.  
Suppose that $\pi$ is a continuous representation
 of $G$
 on a topological vector space $V$.  
Then there is a canonical bijection
\begin{equation}
\label{eqn:Frob}
 \invHom{H}{\pi|_H}{{\mathbb{C}}}
 \simeq 
 \invHom{G}{\pi}{C(G/H)}, 
\qquad
 \lambda \mapsto T
\end{equation}
defined by 
\[
   T(v)(g)=\lambda(\pi(g^{-1})v)
\qquad 
 v \in V.  
\]
Furthermore,
 if $\pi^{\infty}$ is a smooth representation,
 then we have
\[
 \invHom{H}{\pi^{\infty}|_H}{{\mathbb{C}}}
 \simeq 
 \invHom{G}{\pi^{\infty}}{C^{\infty}(G/H)}.  
\]
\end{proposition}
\begin{proof}
The linear map 
$T:V \to C(G/H)$ is continuous 
 because $G \times V \to V$, 
 $(g,v) \mapsto \pi(g^{-1})v$
 is continuous.  
The last statement follows 
 because $G \to V$, 
 $g \mapsto \pi(g)^{-1}v$
 is a $C^{\infty}$-map.  
\end{proof}

\subsection{Admissible representations}
\label{subsec:adm}

In this subsection we review some basic terminologies
 for Harish-Chandra modules.

Let $G$ be a real reductive linear Lie group, 
 and $K$ a maximal compact subgroup of $G$.  
Let ${\mathcal{HC}}$ denote the category
 of Harish-Chandra modules
 where the objects are $({\mathfrak {g}}, K)$-modules
 of finite length,
 and the morphisms are $({\mathfrak {g}}, K)$-homomorphisms.

Let $\pi$ be a continuous representation 
 of $G$ on a Fr{\'e}chet space $V$.  
Suppose 
 that $\pi$ is of finite length,
 namely,
 there are at most finitely many closed $G$-invariant subspaces
 in $V$.  
We say $\pi$ is {\it{admissible}}
 if 
\[
   \dim \operatorname{Hom}_K(\tau, \pi|_K)< \infty
\]
 for any irreducible finite-dimensional representation
 $\tau$ of $K$.  
We denote by $V_K$
 the space of $K$-finite vectors.  
Then $V_K \subset V^{\infty}$
 and the Lie algebra ${\mathfrak {g}}$
 leaves $V_K$ invariant.  
The resulting $({\mathfrak {g}}, K)$-module
 on $V_K$
 is called the underlying $({\mathfrak {g}}, K)$-module
 of $\pi$, 
 and will be denoted by $\pi_K$.

An admissible representation $(\pi,V)$
is said to be {\it{spherical}}
 if $V$ contains a nonzero $K$-fixed vector,
 or equivalently,
 the underlying $({\mathfrak {g}}, K)$-module
 $V_K$ contains a nonzero $K$-fixed vector.

A vector $v \in V$ is said to 
 be {\it{cyclic}}
 if the vector space ${\mathbb{C}}\text{-span} \{\pi(g)v:g \in G\}$
 is dense in $V$.  
If $W$ is a proper $G$-invariant closed subspace
 of $V$, 
 then $v \mod W$ is a cyclic vector
 in the quotient representation 
 on $V/W$. 
For a $K$-finite vector $v$, 
 $v$ is cyclic in $\pi$
 if and only if $v$ is cyclic 
 in the underlying $({\mathfrak {g}}, K)$-module
 $\pi_K$
 in the sense 
 that $U({\mathfrak {g}}_{\mathbb{C}})v =V_K$.

\subsection{Harish-Chandra isomorphism}
\label{subsec:HC}

We review
 the standard normalization
 of the Harish-Chandra isomorphism
 of the ${\mathbb{C}}$-algebra
 $\ZG G$, 
 where we recall from Introduction
 that 
\[
\ZG G
=
U({\mathfrak{g}}_{\mathbb{C}})^G
\equiv
\{
u \in U({\mathfrak{g}}_{\mathbb{C}}):
\operatorname{Ad}(g)u = u
\,\,
\text{ for all }\,\,
g \in G
\}.
\]  
For a connected $G$, 
 $\ZG G$ is equal to the center 
 ${\mathfrak {Z}}({\mathfrak {g}}_{\mathbb{C}})$
 of $U({\mathfrak {g}}_{\mathbb{C}})$.

Let ${\mathfrak{j}}$ be 
 a Cartan subalgebra of 
 ${\mathfrak{g}}$, 
 ${\mathfrak{j}}_{\mathbb{C}}
 =
 {\mathfrak{j}} \otimes_{\mathbb{R}}  
 {\mathbb{C}}
$,
 and
${\mathfrak{j}}_{\mathbb{C}}^{\vee}
 =
 \invHom
 {\mathbb{C}}
 {{\mathfrak{j}}_{\mathbb{C}}}
 {\mathbb{C}}
$. 
We set 
\[
   W({\mathfrak{j}}_{\mathbb{C}})
   :=N_{\widetilde {G}}({\mathfrak{j}}_{\mathbb{C}})
   /Z_{\widetilde {G}}({\mathfrak{j}}_{\mathbb{C}}),
\] 
 where $\widetilde G$ is the group 
 generated by 
 $\operatorname{Ad}(G)$
 and the group $\operatorname{Int}({\mathfrak {g}}_{\mathbb{C}})$
 of inner automorphisms.   
For a connected $G$, 
$W({\mathfrak {j}}_{\mathbb{C}})$
 is the Weyl group
 for the root system
 $\Delta({\mathfrak {g}}_{\mathbb{C}},{\mathfrak {j}}_{\mathbb{C}})$.

Fix a positive system
 $\Delta^+({\mathfrak{g}}_{\mathbb{C}},{\mathfrak{j}}_{\mathbb{C}})$,
 and write ${\mathfrak{n}}_{\mathbb{C}}^+$
 for the sum of the root spaces
 belonging 
 to $\Delta^+(\mathfrak{g}_{\mathbb{C}},{\mathfrak{j}}_{\mathbb{C}})$,
 and ${\mathfrak{n}}_{\mathbb{C}}^-$
 for $\Delta^-(\mathfrak{g}_{\mathbb{C}},{\mathfrak{j}}_{\mathbb{C}})$.  
We set
\begin{equation}
\label{eqn:rhog}
\rho_{{\mathfrak{g}}}
 :=
  \frac 1 2
\sum
_{\alpha
 \in \Delta^+(\mathfrak{g}_{\mathbb{C}},
             {\mathfrak{j}}_{\mathbb{C}})}
\alpha
\in {\mathfrak {j}}_{\mathbb{C}}^{\vee}.  
\end{equation}

Let $\gamma':
U({\mathfrak{g}}_{\mathbb{C}})
\to
U({\mathfrak{j}}_{\mathbb{C}})
\simeq
S({\mathfrak{j}}_{\mathbb{C}})
$
 be the projection
 to the second factor
 of the decomposition
$
U({\mathfrak{g}}_{\mathbb{C}})
=
(
{\mathfrak{n}}_{\mathbb{C}}^-
U({\mathfrak{g}}_{\mathbb{C}})
+
U({\mathfrak{g}}_{\mathbb{C}})
{\mathfrak{n}}_{\mathbb{C}}^+)
\oplus U({\mathfrak{j}}_{\mathbb{C}})
$. 
Then we have the Harish-Chandra isomorphism
\begin{equation}
\label{eqn:HCisom}
\ZG G
=
U({\mathfrak{g}}_{\mathbb{C}})^G
\underset \gamma{\overset\sim\to}
S({\mathfrak{j}}_{\mathbb{C}})
^{W({\mathfrak{j}}_{\mathbb{C}})},
\end{equation}
where 
$
  \gamma :
U({\mathfrak{g}}_{\mathbb{C}})
\to
S({\mathfrak{j}}_{\mathbb{C}})
$
is defined by
$
\langle \gamma(u), \lambda\rangle
=
\langle \gamma'(u), \lambda - \rho_{{\mathfrak{g}}}\rangle
$
 for all
 $
\lambda \in {\mathfrak{j}}_{\mathbb{C}}^{\vee}.  
$

Then any element 
 $\lambda \in {\mathfrak {j}}_{\mathbb{C}}^{\vee}$
 gives a ${\mathbb{C}}$-algebra
 homomorphism
 $\chi_{\lambda}: \ZG G \to {\mathbb{C}}$
 via the isomorphism
 \eqref{eqn:HCisom}, 
 and $\chi_{\lambda}=\chi_{\lambda'}$
 if and only if
 ${\lambda}'= w{\lambda}$
 for some $w \in W({\mathfrak {j}}_{\mathbb{C}})$.  
This correspondence 
 yields a bijection:
\begin{equation}
\label{eqn:ZG}
  \HomCalg{\ZG G}{{\mathbb{C}}}
  \simeq
  {\mathfrak {j}}_{{\mathbb{C}}}^{\vee}/W({\mathfrak{j}}_{\mathbb{C}}), 
\quad
  \chi_{\lambda} \leftrightarrow \lambda.  
\end{equation}

In our normalization,
 the $\ZG G$-infinitesimal character 
 of the trivial representation ${\bf{1}}$
 of $G$ is given by $\rho_{{\mathfrak {g}}}$.

For $\lambda \in {\mathfrak {j}}_{\mathbb{C}}^{\vee}
/W({\mathfrak {j}}_{\mathbb{C}})$, 
 we set
\begin{align*}
C^{\infty}(G;\chi_{\lambda}^R)
:=&
\{f \in C^{\infty}(G)
:
R_u f = \chi_{\lambda}(u) f
\,\,\text{ for any }\,\, u \in \ZG G\}, 
\\
C^{\infty}(G;\chi_{\lambda}^L)
:=&
\{f \in C^{\infty}(G)
:
L_u f = \chi_{\lambda}(u) f
\,\,\text{ for any }\,\, u \in \ZG G\}.  
\end{align*}
Then we have 
 $C^{\infty}(G;\chi_{\lambda}^R)
 =
 C^{\infty}(G;\chi_{-\lambda}^L)$.

Let $H$ be a closed subgroup
 of $G$.  
Since the action of $\ZG G$
 on $C^{\infty}(G)$
 via $R$ (and via $L$)
 commutes with the right $H$-action,
$R_u$ and $L_u$
$(u \in \ZG G)$ 
induce differential operators 
 on $G/H$.  
Thus, 
 for $\lambda \in {\mathfrak {j}}_{\mathbb{C}}^{\vee}
/W({\mathfrak {j}}_{\mathbb{C}})$, 
 we can define 
\begin{align*}
C^{\infty}(G/H;\chi_{\lambda}^R)
:=&
\{f \in C^{\infty}(G/H)
  :
  R_u f = \chi_{\lambda}(u) f
\,\,\text{ for any } \,\, u \in \ZG G\}, 
\\
C^{\infty}(G/H;\chi_{\lambda}^L)
:=&
\{f \in C^{\infty}(G/H)
  :
  L_u f = \chi_{\lambda}(u) f
\,\,\text{ for any } \,\, u \in \ZG G\}.  
\end{align*}

\subsection{Shintani functions of moderate growth}
\label{subsec:mod}

Without loss of generality,
 we may and do assume 
 that a real reductive linear Lie group $G$ is realized as a closed subgroup 
 of $GL(n,{\mathbb{R}})$
 such that $G$ is stable under the transpose
 of matrix $g \mapsto {}^{t\!} g$
 and $K=O(n) \cap G$.  
For $g \in G$ we define a map $\| \cdot \|
: G \to {\mathbb{R}}$
 by 
\[
\|g\|
:=\| g \oplus {}^t g^{-1}\|_{\operatorname{op}}
\]
where $\| \cdot \|_{\operatorname{op}}$
 is the operator norm of $M(2n,{\mathbb{R}})$.  
A continuous representation $\pi$ of $G$ 
on a Fr{\'e}chet space $V$
 is said to be of {\it{moderate growth}}
 if for each continuous semi-norm
 $|\cdot |$ on $V$ 
 there exist a continuous semi-norm
 $|\cdot |'$ on $V$
 and a constant $d \in {\mathbb{R}}$
 such that
\[
   |\pi(g)u| \le \| g \|^d |u|'
\quad
\text{for }
g \in G, u \in V.  
\]
For any admissible representation $(\pi, {\mathcal{H}})$
 such that ${\mathcal{H}}$ is a Banach space, 
 the smooth representation $(\pi^{\infty}, {\mathcal{H}}^{\infty})$
 has moderate growth.  
We say $(\pi^{\infty}, {\mathcal{H}}^{\infty})$
 is an {\it{admissible smooth representation}}.  
By the Casselman--Wallach globalization theory, 
there is a canonical equivalence of categories
 between the category ${\mathcal{HC}}$
 of $({\mathfrak {g}}, K)$-modules
 of finite length 
 and the category of admissible smooth representations 
 of $G$ (\cite[Chapter 11]{WaI}).  
In particular,
 the Fr{\'e}chet representation
 $\pi^{\infty}$ is uniquely 
 determined by its underlying
 $({\mathfrak {g}}, K)$-module.  
We say $\pi^{\infty}$
 is the {\it{smooth globalization}}
 of $\pi_K \in {\mathcal{HC}}$.

For simplicity,
 by an {\it{irreducible smooth representation}}
we shall mean an irreducible admissible smooth representation
 of $G$.

\begin{definition}
\label{def:mod}
{\rm{
A smooth function $f$ on $G$
 is said to have {\it{moderate growth}}
 if $f$ satisfies the following three properties:
\begin{enumerate}
\item[(1)]
$f$ is right $K$-finite.  
\item[(2)]
$f$ is $\ZG G$-finite.  
\item[(3)]
There exists a constant $d \in {\mathbb{R}}$
 (depending on $f$)
 such that if $u \in U({\mathfrak{g}}_{\mathbb{C}})$
 then there exists $C \equiv C(u)$
 satisfying 
\[
  |(R_u f)(x)| \le C\| x \|^d
  \qquad
  (x \in G).  
\]
\end{enumerate}
We denote by $C_{\operatorname{mod}}^{\infty}(G)$
 the space of all $f \in C^{\infty}(G)$
 having moderate growth.  
}}
\end{definition}
If $(\pi, V)$ is an admissible representation
 of moderate growth,
 then the matrix coefficient 
 $G \to {\mathbb{C}}$, 
 $g \mapsto \langle \pi(g) v, u\rangle$
 belongs to $C_{\operatorname{mod}}^{\infty}(G)$
 for any $v \in V_K$
 and any linear functional $u$
 of the Fr{\'e}chet space $V$.

We define the space
 of Shintani functions
 of moderate growth by 
\begin{equation}
\label{eqn:Shmod}
\operatorname{Sh}_{\operatorname{mod}}(\lambda,\nu)
:=
\operatorname{Sh}(\lambda,\nu)
\cap 
C_{\operatorname{mod}}^{\infty}(G).  
\end{equation}

\section{Finite-multiplicity properties
 of branching laws}
\label{sec:Shfin}

We are ready to make a precise statement
 of Theorem \ref{thm:A}, 
 and enrich it 
 by adding some more equivalent conditions.  
The main results
 of this section is Theorem \ref{thm:dozen}.  

\subsection{Finite-multiplicity properties 
 of branching laws}
\label{subsec:dozen}
\begin{theorem}
\label{thm:dozen}
The following twelve conditions
 on a pair of real reductive algebraic groups 
 $G \supset G'$
 are equivalent:
\begin{enumerate}
\item[{\rm{(i)}}]
{\rm{(PP)}}\enspace
There exist minimal parabolic subgroups
 $P$ and $P'$
 of $G$ and $G'$, 
respectively,
 such that $P P'$ is open in $G$.  

\item[{\rm{(ii)}}]
$(\operatorname{Sh})$\enspace
$\dim_{\mathbb{C}} \operatorname{Sh} (\lambda,\nu) < \infty$
 for any pair $(\lambda,\nu)$
 of $(\ZG G, \ZG {G'})$-infinitesimal characters.  

\item[{\rm{(iii)}}]
$(\operatorname{Sh}_{\operatorname{mod}})$\enspace
$\dim_{\mathbb{C}} \operatorname{Sh}_{\operatorname{mod}} (\lambda,\nu) < \infty$
 for any pair $(\lambda,\nu)$
 of $(\ZG G, \ZG {G'})$-infinitesimal characters.  

\item[{\rm{(iv)}}]
$(\operatorname{Sh}_{\operatorname{mod}})_{\bf{1}}$
\enspace
$\dim_{\mathbb{C}} \operatorname{Sh}_{\operatorname{mod}} (\rho_{{\mathfrak {g}}}, \rho_{{\mathfrak {g}}'}) < \infty$.  

\item[{\rm{(v)}}]
$(\infty \downarrow)$
\enspace
$\dim_{\mathbb{C}} \invHom{G'}{\pi^{\infty}|_{G'}}{\tau^{\infty}}< \infty$ 
 for any pair $(\pi^{\infty}, {\tau^{\infty}})$
 of admissible smooth representations
 of $G$ and $G'$.  

\item[{\rm{(vi)}}]
$(\infty \downarrow)_K$
\enspace
$\dim_{\mathbb{C}} \invHom{G'}{\pi^{\infty}|_{G'}}{\tau^{\infty}}< \infty$
 for any pair $(\pi^{\infty}, {\tau^{\infty}})$
 of admissible smooth representations
 of $G$ and $G'$
 such that $\pi^{\infty}$ and $(\tau^{\infty})^{\vee}$
 have cyclic spherical vectors.  

\item[{\rm{(vii)}}]
$({\mathcal{H}}\downarrow)$
\enspace
$\dim_{\mathbb{C}} \invHom{G'}{\pi|_{G'}}{\tau}< \infty$
 for any pair $(\pi, \tau)$
 of admissible Hilbert representations
 of $G$ and $G'$.  

\item[{\rm{(viii)}}]
$({\mathcal{H}}\downarrow)_K$
\enspace
$\dim_{\mathbb{C}} \invHom{G'}{\pi|_{G'}}{\tau}< \infty$
 for any pair $(\pi, \tau)$
 of admissible Hilbert representations
 of $G$ and $G'$
 such that $\pi$ and $\tau^{\vee}$
 have cyclic spherical vectors.  

\item[{\rm{(ix)}}]
$(\infty \otimes)$
\enspace
$\dim_{\mathbb{C}} \invHom{G'}{\pi^{\infty} \otimes \tau^{\infty}} {\mathbb{C}}< \infty$
 for any pair $(\pi^{\infty}, {\tau^{\infty}})$
 of admissible smooth representations
 of $G$ and $G'$.    

\item[{\rm{(x)}}]
$(\infty \otimes)_K$
\enspace
$\dim_{\mathbb{C}} \invHom{G'}{\pi^{\infty} \otimes \tau^{\infty}} {\mathbb{C}}< \infty$
 for any pair $(\pi^{\infty}, {\tau^{\infty}})$
 of admissible smooth representations
 of $G$ and $G'$
 such that $\pi^{\infty}$ and $\tau^{\infty}$
 have cyclic spherical vectors.  

\item[{\rm{(xi)}}]
$({\mathcal{H}} \otimes)$
\enspace
$\dim_{\mathbb{C}} \invHom{G'}{\pi \otimes \tau} {\mathbb{C}}< \infty$
 for any pair $(\pi, \tau)$
 of admissible Hilbert representations
 of $G$ and $G'$.  

\item[{\rm{(xii)}}]
$({\mathcal{H}} \otimes)_K$
\enspace
$\dim_{\mathbb{C}} \invHom{G'}{\pi \otimes \tau} {\mathbb{C}}< \infty$
 for any pair $(\pi, \tau)$
 of admissible Hilbert representations
 of $G$ and $G'$
 such that $\pi$ and $\tau$
 have cyclic spherical vectors.  
\end{enumerate}
\end{theorem}

\subsection{Outline of the proof of Theorem \ref{thm:dozen}}
\label{subsec:outline}

The following implications are obvious:
\[
{\rm{(ii)}}\enspace(\operatorname{Sh})
\,\,\Rightarrow\,\,
{\rm{(iii)}}\enspace
(\operatorname{Sh}_{\operatorname{mod}})
\,\,\Rightarrow\,\,
{\rm{(iv)}}\enspace
(\operatorname{Sh}_{\operatorname{mod}})_{\bf{1}}.  
\]

By Lemma \ref{lem:rest}, 
 we have the following inclusive relations
 and isomorphism.  
\[
\invHom{G'}{\pi^{\infty} \otimes(\tau^{\vee})^{\infty}}{{\mathbb{C}}}
\supset
\invHom{G'}{\pi^{\infty}}{\tau^{\infty}}
\supset 
\invHom{G'}{\pi} {\tau}
\simeq 
\invHom{G'}{\pi \otimes \tau^{\vee}}{{\mathbb{C}}}.  
\]
In turn, 
 we have the obvious implications
 and equivalences as below.  
\begin{alignat*}{7}
&{\rm{(ix)}}\enspace (\infty \otimes) 
&&\,\,\Longrightarrow\,\,
&&{\rm{(v)}}\enspace (\infty \downarrow)
&&\,\,\Longrightarrow\,\, 
&& {\rm{(vii)}}\enspace ({\mathcal{H}}\downarrow)
&&\,\,\Longleftrightarrow\,\, 
&&{\rm{(xi)}}\enspace({\mathcal{H}}\otimes)
\\
& \Downarrow && &&\Downarrow && && \Downarrow && && \Downarrow
\\
&{\rm{(x)}}\enspace (\infty \otimes)_K
&&\,\,\Longrightarrow\,\, 
&&{\rm{(vi)}}\enspace (\infty \downarrow)_K
&&\,\,\Longrightarrow\,\, 
&& {\rm{(viii)}}\enspace ({\mathcal{H}}\downarrow)_K
&&\,\,\Longleftrightarrow\,\, 
&&{\rm{(xii)}}\enspace({\mathcal{H}}\otimes)_K.  
\end{alignat*}

The remaining non-trivial implications
 are 
\begin{eqnarray*}
&  {\rm{(viii)}}\enspace({\mathcal{H}}\downarrow)_K
\,\,\text{ or }\,\,
  {\rm{(iv)}}\enspace(\operatorname{Sh}_{\operatorname{mod}})_{\bf{1}}
\\
&
\Downarrow
\\
&  {\rm{(i)}}\enspace {\rm{(PP)}}
\\
&\Downarrow 
\\
&  {\rm{(ii)}}\enspace(\operatorname{Sh})
\,\,\text{ and }\,\,
  {\rm{(ix)}}\enspace(\infty \otimes).  
\end{eqnarray*}
We discuss the geometric property (PP)
 in Section \ref{subsec:PP}.  
Then the implications
\[
  {\rm{(i)}}\enspace\text{(PP)}
  \Rightarrow {\rm{(ii)}}\enspace\text{(Sh)} 
  \text{ and }{\rm{(ix)}}\enspace (\infty \otimes)
\]
 are given in Propositions \ref{prop:PPSh} and 
 \ref{prop:PPtensor}, 
 respectively.

The implication
\[
  {\rm{(viii)}}\enspace({\mathcal{H}}\downarrow)_K 
  \Rightarrow 
  {\rm{(i)}}\enspace \text{(PP)}
\]
 is proved in Proposition \ref{prop:L}, 
 and the implication 
\[
  {\rm{(iv)}}\enspace(\operatorname{Sh}_{\operatorname{mod}})_{\bf{1}}
   \Rightarrow
  {\rm{(i)}}\enspace \text{(PP)}
\]
 is proved in Corollary \ref{cor:Shtriv}.

The relationship of $\invHom{G'}{\pi^{\infty}}{\tau^{\infty}}$
 (symmetry breaking operators)
 and $\operatorname{Sh}(\lambda,\nu)$
 (Shintani functions)
 will be discussed in Sections \ref{sec:break} and \ref{sec:PS}.  

\subsection{Invariant trilinear forms}
\label{subsec:tri}
Suppose that $\pi_i^{\infty}$
 are admissible smooth representations
 of a Lie group $G$ 
 on Fr{\'e}chet spaces ${\mathcal{H}}_i^{\infty}$
 $(i=1,2,3)$.  
A continuous trilinear form
\[
 T: {\mathcal{H}}_1^{\infty}
    \times {\mathcal{H}}_2^{\infty}
    \times {\mathcal{H}}_3^{\infty}
    \to 
    {\mathbb{C}}
\]
 is {\it{invariant}}
 if 
\[
  T(\pi_1^{\infty}(g)u_1, 
    \pi_2^{\infty}(g)u_2,
    \pi_3^{\infty}(g)u_3)
 =T(u_1,u_2,u_3)
\quad
\text{ for all }\,\,
 g \in G \,\,
 \text{ and }\,\,
 u_i \in {\mathcal{H}}_i^{\infty}
\,\,
 (i=1,2,3).  
\]
\begin{corollary}
\label{cor:tri}
\begin{enumerate}
\item[{\rm{1)}}]
Suppose $G$ is a real reductive Lie group.  
Then the following four conditions on $G$ are equivalent:
\begin{enumerate}
\item[{\rm{(i)}}]
$(G \times G \times G)/\operatorname{diag} G$
 is real spherical
 as a $(G \times G \times G)$-space.  
\item[{\rm{(ii)}}]
{\rm{(Shintani functions in the group case)}}\enspace
The space $\operatorname{Sh}((\lambda_1, \lambda_2), \lambda_3)$
 of Shintani functions
 for $(G \times G, \operatorname{diag} G)$
 is finite-dimensional 
 for any triple
 of $\ZG G$-infinitesimal characters
 $\lambda_1$, $\lambda_2$ and $\lambda_3$.  
\item[{\rm{(iii)}}]
{\rm{(Symmetry breaking for the tensor product)}}\enspace
For any triple of admissible smooth representations
 $\pi_1^{\infty}$, $\pi_2^{\infty}$, 
 and $\pi_3^{\infty}$ of $G$, 
\[
\dim_{\mathbb{C}} \operatorname{Hom}_G(\pi_1^{\infty} \otimes \pi_2^{\infty}, 
\pi_3^{\infty})< \infty.  
\]
\item[{\rm{(iv)}}]
{\rm{(Invariant trilinear form)}}\enspace
For any triple of admissible smooth representations
 $\pi_1^{\infty}$, $\pi_2^{\infty}$ and $\pi_3^{\infty}$ of $G$, 
 the space of invariant trilinear forms
 is finite-dimensional.  
\end{enumerate}
\item[{\rm{2)}}]
Suppose that $G$ is a simple Lie group.  
Then one of {\rm{(}}therefore any of{\rm{)}} the above four 
 equivalent conditions
 is fulfilled
 if and only if either $G$ is compact
 or ${\mathfrak {g}}$ is isomorphic to ${\mathfrak{o}}(n,1)$
 $(n \ge 2)$.  
\end{enumerate}
\end{corollary}

\begin{proof}
The first and second statements
 are special cases
 of Theorems \ref{thm:dozen}
 and \ref{thm:PP}, 
 respectively.  
\end{proof}
\begin{remark}
\label{rem:triple}
{\rm{
As in (vi) and (viii) of Theorem \ref{thm:dozen}, 
 the conditions (iii) and (iv)
 of Corollary \ref{cor:tri}
 are equivalent to the analogous statements
 by replacing $\pi_j^{\infty}$
 ($j=1,2,3$)
 with spherical ones.  
}}
\end{remark}
\begin{remark}
{\rm{
The equivalence (i) $\Leftrightarrow$ (ii)
 was first formulated
 in \cite{xtoshi95}
 with a sketch of proof.  
}}
\end{remark}
\begin{example}
\label{ex:gBR}
{\rm{
For $G=O(n,1)$, 
 a meromorphic family
 of invariant trilinear forms
 for spherical principal series
 representations
 was constructed in \cite{CKOP}.  
}}
\end{example}

\section{Real spherical manifolds and Shintani functions}
\label{sec:sph}
In this section,
 we regard Shintani functions
 as smooth functions
 on the homogeneous space
 $(G \times G')/\operatorname{diag}G'$, 
 and apply the theory of real spherical 
 homogeneous spaces
 \cite{xtoshitoshima}.  
In particular,
 we give a proof of the implication
 (i) (PP) $\Rightarrow$
 (ii) (Sh) 
 and 
 (ix) ($\infty \otimes$)
 in Theorem \ref{thm:dozen}
 (see Proposition \ref{prop:PPSh}).  
\subsection{Real spherical 
 homogeneous spaces 
 and (PP)}
\label{subsec:PP}
A complex manifold $X_{\mathbb{C}}$ 
 with action of a complex reductive group $G_{\mathbb{C}}$
 is called {\it{spherical}} 
 if a Borel subgroup of $G_{\mathbb{C}}$
 has an open orbit in $X_{\mathbb{C}}$.  
In the real setting, 
 in search of a good framework for global analysis 
 on homogeneous spaces
 which are broader than the usual 
 ({\it{e.g.}} symmetric spaces),
 we proposed to call:
\begin{definition}
[{\cite{xtoshi95}}]
\label{def:realsp}
\rm{
Let $G$ be a real reductive Lie group.  
We say a smooth manifold $X$ with $G$-action
 is {\it{real spherical}}
 if a minimal parabolic subgroup $P$
 of $G$ has an open orbit in $X$.  
}
\end{definition}
The significance of this geometric property
is its application
 to the finite-multiplicity property
 in the regular representation
 of $G$ on $C^{\infty}(X)$, 
 which was proved by using the theory 
 of hyperfunctions
 and regular singularities
 of a system of partial differential equations:
\begin{proposition}
[{\cite[Theorem A and Theorem 2.2]{xtoshitoshima}}]
\label{prop:HP}
Suppose $G$ is a real reductive linear Lie group,
 and $H$ is a closed subgroup.  
If the homogeneous space $G/H$ 
 is real spherical,
 then the regular representation of $G$
 on the Fr{\'e}chet space
 $C^{\infty}(G/H;\chi_{\lambda}^L)$
 is admissible
 for any $\ZG G$-infinitesimal character 
 $\lambda \in {\mathfrak {j}}_{\mathbb{C}}^{\vee}
 /W({\mathfrak {j}}_{\mathbb{C}})$.  
In particular,
\[
\text{
$\invHom{G}{\pi^{\infty}}{C^{\infty}(G/H)}$
 is finite-dimensional
}
\]
for any smooth admissible representation
 $\pi^{\infty}$ of $G$.  
\end{proposition}

Suppose that $G'$ is an algebraic reductive subgroup of 
 $G$.  
Let $P'$ be a minimal parabolic subgroup
 of $G'$.  

\begin{definition}
[{\cite{xtoshitoshima}}]
\label{def:pp}
\rm{
We say the pair $(G,G')$ satisfies (PP)
 if one of the following five equivalent
 conditions is satisfied.  
\begin{enumerate}
\item[{\rm{(PP1)}}]
$(G \times G')/\operatorname{diag}G'$ is real spherical
 as a $(G \times G')$-space.  

\item[{\rm{(PP2)}}]
$G/P'$ is real spherical as a $G$-space.  

\item[{\rm{(PP3)}}]
$G/P$ is real spherical as
 a $G'$-space. 

\item[{\rm{(PP4)}}]
$G$ has an open orbit in $G/P \times G/P'$
 via the diagonal action.  

\item[{\rm{(PP5)}}]
There are finitely many $G$-orbits in $G/P \times G/P'$
 via the diagonal action.  
\end{enumerate}
}
\end{definition}
The above five equivalent conditions are determined 
 only by the Lie algebras ${\mathfrak {g}}$ and ${\mathfrak {g}}'$.  
Therefore we also say
 that the pair $({\mathfrak {g}}, {\mathfrak {g}}')$
 of Lie algebras
 satisfies (PP).  

\vskip 1pc
Next we consider another property, 
 to be denoted by (BB), 
 which is stronger than (PP).  
Let $G_{\mathbb{C}}$ be a complex Lie group
 with Lie algebra 
 ${\mathfrak{g}}_{\mathbb{C}}={\mathfrak{g}}\otimes_{\mathbb{R}}{\mathbb{C}}$, 
 and $G_{\mathbb{C}}'$ a subgroup of $G_{\mathbb{C}}$
 with complexified Lie algebra 
 ${\mathfrak{g}}_{\mathbb{C}}'
 ={\mathfrak{g}}'\otimes_{\mathbb{R}}{\mathbb{C}}$.  
Let $B$ and $B'$ be Borel subgroups of $G_{\mathbb{C}}$ 
 and $G_{\mathbb{C}}'$, 
 respectively.  

\begin{definition}
\label{def:BB}
{\rm{
We say the pair
 $(G,G')$
 (or the pair $({\mathfrak{g}}, {\mathfrak{g}}')$)
 satisfies {\rm{(BB)}}
 if one of the following five equivalent conditions is satisfied:
\begin{enumerate}
\item[{\rm{(BB1)}}]
$(G_{\mathbb{C}} \times G_{\mathbb{C}}')/
\operatorname{diag}G_{\mathbb{C}}'$
 is spherical 
 as a $(G_{\mathbb{C}} \times G_{\mathbb{C}}')$-space.   
\item[{\rm{(BB2)}}]
$G_{\mathbb{C}}/B'$ is spherical 
 as a $G_{\mathbb{C}}$-space.  

\item[{\rm{(BB3)}}]
 $G_{\mathbb{C}}/B$ is real spherical as
 a $G_{\mathbb{C}}'$-space.

\item[{\rm{(BB4)}}]
$G_{\mathbb{C}}$ has an open orbit in $G_{\mathbb{C}}/B \times G_{\mathbb{C}}/B'$
 via the diagonal action.  
\item[{\rm{(BB5)}}]
There are finitely many $G_{\mathbb{C}}$-orbits
 in $G_{\mathbb{C}}/B \times G_{\mathbb{C}}/B'$
 via the diagonal action.  
\end{enumerate}
}}
\end{definition}
The above five equivalent conditions 
 are determined 
 only by the complexified Lie algebras
 ${\mathfrak {g}}_{\mathbb{C}}$
 and 
 ${\mathfrak {g}}_{\mathbb{C}}'$.  
It follows from \cite[Lemmas 4.2 and 5.3]{xtoshitoshima}
 that we have an implication
\[
  \text{(BB)} \Rightarrow \text{(PP)}.  
\]

\subsection{Shintani functions
 and real spherical homogeneous spaces}
\label{subsec:ShPP}

We return to Shintani functions
 for the pair $G \supset G'$.  
Let $(\lambda, \nu)
 \in {\mathfrak {j}}_{\mathbb{C}}^{\vee}
 /W({\mathfrak {j}}_{\mathbb{C}})
\times  ({\mathfrak {j}}_{\mathbb{C}}')^{\vee}
 /W({\mathfrak {j}}_{\mathbb{C}}')$.  
We begin 
 with an elementary and useful point
 of view:

\begin{lemma}
\label{lem:ShPP}
The multiplication map
\[
\varphi:G \times G' \to G,
\qquad
        (g,h) \mapsto g h^{-1}
\]
induces the following linear isomorphism
\[
   \varphi^{\ast}:
   \operatorname{Sh} (\lambda,\nu)
   \overset \sim \to 
   \invHom{K \times K'}
   {{\bf{1}}\boxtimes{\bf{1}}}
   {C^{\infty}
    ((G \times G')/\operatorname{diag}G';\chi_{\lambda,\nu}^L)}, 
\]
where ${\bf{1}}$ denotes
 the trivial one-dimensional representation
 of the group $K$
 (or that of $K'$).  
\end{lemma}
\begin{proof}
The pull-back of functions
\[
   \varphi^{\ast}:
   C^{\infty}(G)
   \overset \sim \to 
   C^{\infty}((G \times G')/\operatorname{diag}G') 
\]
satisfies
\[
  L_X L_Y(\varphi^{\ast} f)
  =\varphi^{\ast}(L_X R_Y f)
\qquad
\text{for all
 $X \in {\mathfrak {g}}$, 
 $Y \in {\mathfrak {g}}'$
 and $f \in C^{\infty}(G)$.  }
\]
Hence $\varphi^{\ast}$ maps 
 $\operatorname{Sh} (\lambda,\nu)$
 onto the space 
 of $(K \times K')$-invariant functions
 of $C^{\infty}
    ((G \times G')/\operatorname{diag}G';\chi_{\lambda,\nu}^L)$.  
\end{proof}

\begin{proposition}
\label{prop:PPSh}
If $(G,G')$ satisfies {\rm{(PP)}}, 
 then $\dim_{\mathbb{C}}\operatorname{Sh} (\lambda,\nu)<\infty$
 for any pair $(\lambda,\nu)$
 of $(\ZG G, \ZG {G'})$-infinitesimal characters.  
\end{proposition}
\begin{proof}
Since $(G,G')$ satisfies (PP1), 
 the regular representation
 on the Fr{\'e}chet space
$
   C^{\infty}
    ((G \times G')/\operatorname{diag}G';\chi_{\lambda,\nu}^L)
$ 
 is admissible 
 as a representation 
of the direct product group $G \times G'$
 by Proposition \ref{prop:HP}.  
Therefore,
 Proposition \ref{prop:PPSh}
 follows from Lemma \ref{lem:ShPP}.  
\end{proof}

\begin{proposition}
\label{prop:PPtensor}
If $(G,G')$ satisfies {\rm{(PP)}}, 
 then $\invHom{G'}{\pi^{\infty} \otimes \tau^{\infty}}{{\mathbb{C}}}$
 is finite-dimensional 
 for any pair $(\pi^{\infty}, \tau^{\infty})$
 of admissible smooth representations 
 of $G$ and $G'$.  
\end{proposition}
\begin{proof}
Since $(G \times G')/\operatorname{diag}G'$ is real spherical,
\[
   \dim_{\mathbb{C}}
   \invHom{G \times G'}
          {\pi^{\infty} \boxtimes \tau^{\infty}}
          {C^{\infty}(G \times G'/\operatorname{diag}G')}<\infty
\]
by Proposition \ref{prop:HP}.  
Therefore Proposition \ref{prop:PPtensor}
 follows from the Frobenius reciprocity
 (Proposition \ref{prop:Frob}).  
\end{proof}

\section{Construction of intertwining operators}
\label{sec:Poisson}

In this section
 we give lower bounds 
 of the dimension of the space of symmetry breaking operators
 for the restriction of admissible Hilbert representations.

\subsection{A generalization of the Poisson integral transform}
\label{subsec:Poisson}
We fix some general notation.  
Let $H$ be a closed subgroup of $G$.  
Given a finite-dimensional representation $\tau$
 of $H$
 on a vector space $W_{\tau}$, 
 we denote by ${\mathcal{W}}_{\tau}$
 the $G$-equivariant vector bundle
 $G \times_H W_{\tau}$
 over the homogeneous space $G/H$.  
Then we have a representation
 of $G$ naturally 
 on the space of sections
\[
{\mathcal{F}}(G/H;\tau)
\equiv
{\mathcal{F}}(G/H;{\mathcal{W}}_{\tau})
\simeq
\{f \in {\mathcal{F}}(G) \otimes W:
  f (\cdot h) = \tau(h)^{-1}f(\cdot)
\,\,\text{ for }\,\,h \in H\},
\]
where ${\mathcal{F}}={\mathcal{A}}$, 
 $C^{\infty}$, ${\mathcal{D}}'$, 
 or ${\mathcal{B}}$
 denote the sheaves
 of analytic functions,
 smooth functions,
 distributions,
 or hyperfunctions,
respectively.

\begin{remark}
{\rm{
We shall regard distributions
 as generalized functions
 {\`a} la Gelfand
 (or a special case of hyperfunctions {\`a} la Sato)
 rather than continuous linear forms
 on $C_c^{\infty}(G/H, {\mathcal{W}}_{\tau})$.  
}}
\end{remark}
We define a one-dimensional representation
 of $H$ by 
\[
  \chi_{G/H}:H \to {\mathbb{R}}^{\times}, 
  \quad
  h \mapsto |\det(\operatorname{Ad}_{\#}(h)
             : {\mathfrak {g}}/{\mathfrak {h}}
              \to {\mathfrak {g}}/{\mathfrak {h}})|^{-1}, 
\]
where $\operatorname{Ad}_{\#}(h)$ is the quotient representation
 of the adjoint representation 
 $\operatorname{Ad}(h) \in GL_{{\mathbb{R}}}({\mathfrak {g}})$.  
The bundle of volume densities
 of $X=G/H$
 is given as a $G$-homogeneous line bundle
 $\Omega_X \simeq G \times_H \chi_{G/H}$.  
Then the dualizing bundle
 of ${\mathcal{W}}_{\tau}$ is given, 
 as a homogeneous vector bundle,
 by 
\[
   {\mathcal{W}}_{\tau}^{\ast}
  :=
   (G \times_H W_{\tau}^{\vee})\otimes \Omega_X
  \simeq
  G \times_H \tau^{\ast}, 
\]
 where $(\tau^{\vee}, W_{\tau}^{\vee})$
 denotes the contragredient representation 
 of $(\tau, W_{\tau})$, 
 and $\tau^{\ast}$ is a complex representation 
 of $H$ given by 
\begin{equation}
\label{eqn:dual}
   \tau^{\ast}
   :=
   \tau^{\vee} \otimes \chi_{G/H}.  
\end{equation}

Suppose now
 that $Q$ is a parabolic subgroup
 of a real reductive Lie group $G$, 
and $Q=LN$ a Levi decomposition.  
By an abuse of notation
 we write ${\mathbb{C}}_{2 \rho}$
 for $\chi_{G/Q}$.  
Then ${\mathbb{C}}_{2 \rho}$ is trivial
 on the nilpotent subgroup $N$, 
 and the restriction of ${\mathbb{C}}_{2 \rho}$
 to the Levi part $L$
 coincides with the one-dimensional representation
 defined by
\[
  L \to {\mathbb{R}}^{\times}
  \quad
  l \mapsto |\det(\operatorname{Ad}(l)
             : {\mathfrak {n}} \to {\mathfrak {n}})|.  
\]

In view of the isomorphism
 $K/(Q \cap K) \overset \sim \to G/Q$, 
 ${\mathcal{W}}_{\tau}$ may be regarded
 as a $K$-equivariant vector bundle
 over $K/(Q \cap K)$.  
Then there exist a $K$-invariant Hermitian vector bundle
 structure on ${\mathcal{W}}_{\tau}$
 and a $K$-invariant Radon measure
 on $K/(Q \cap K)$, 
 and we can define 
 a Hilbert representation of $G$ 
 on the Hilbert space $L^2(G/Q;\tau)$
 of square integrable sections of ${\mathcal{W}}_{\tau}$.  
%In our normalization,
% this Hilbert representation 
% is a unitary representation of $G$
% if $\tau \simeq \tau^{\ast}$
% as representations of $Q$.  
The underlying $({\mathfrak {g}}, K)$-module
 of ${\mathcal{F}}(G/Q;\tau)$
 does not depend on the choice
 of ${\mathcal{F}}={\mathcal{A}}, 
 C^{\infty}$, 
 ${\mathcal{D}}'$, 
 ${\mathcal{B}}$, 
 or $L^2$, 
 and will be denoted 
 by $E(G/Q;\tau)$.

We denote by $\widehat G_f$
 and $\widehat L_f$
 the sets of equivalence classes 
 of finite-dimensional irreducible representations
 over ${\mathbb{C}}$
 of the groups $G$ and $L$, 
respectively.  
Then there is an injective map
\[
\widehat G_f \hookrightarrow \widehat L_f, 
\qquad
 \sigma \mapsto \lambda(\sigma)
\]
 such that 
%\footnote
%{
%   $\lambda(\sigma)=H_0({\mathfrak {n}},\sigma) \otimes {\mathbb{C}}_{2 \rho}$
%}
$\sigma$ is the unique quotient of 
 the $({\mathfrak {g}}, K)$-module 
 $E(G/Q;\lambda(\sigma))$.  
We note that $\lambda({\bf{1}}) = {\mathbb{C}}_{2\rho}$.

Here is a Hilbert space
 analog of \cite[Theorem 3.1]{xtoshitoshima}
 which was formulated in the category 
 of $({\mathfrak {g}}, K)$-modules 
 (and was proved in the case
 where $Q$ is a minimal parabolic subgroup of $G$).  
\begin{proposition}
\label{prop:N}
Let $Q$ be a parabolic subgroup of $G$, 
 and $H$ a closed subgroup of $G$.  
Suppose that
 there are $m$ disjoint $H$-invariant open subsets
 in the real generalized flag variety $G/Q$.  
Then 
\[
 \dim \invHom {G}{L^2(G/Q;\lambda(\sigma))}{C(G/H;\tau)}
  \ge
  m \dim \invHom {H} {\sigma|_H}{\tau}, 
\]
for any finite-dimensional representations 
 $\sigma$ and $\tau$ of $G$ and $H$, 
 respectively.  
In particular,
 we have
\[
  \dim \invHom G {L^2(G/Q, \Omega_{G/Q})}{C(G/H)} \ge m.  
\]
\end{proposition}
A key of the proof is the construction
 of integral intertwining operators
 formulated as follows:
\begin{proposition}
\label{prop:Poisson-1}
Let $\tau$ and $\zeta$ be finite-dimensional representations
 of $H$ and $Q$, 
 respectively.  
We set $\zeta^{\ast} = \zeta^{\vee} \otimes {\mathbb{C}}_{2\rho}$.  
Let $({\mathcal{F}}, {\mathcal{F}}')$
 be one of the pairs
\[
 ({\mathcal{A}},{\mathcal{B}}), \,\,
 (C^{\infty},{\mathcal{D}}'),\,\,
 (L^2,L^2), \,\,
 ({\mathcal{D}}',C^{\infty}),\,\,
  \text{ or }\,\,
 ({\mathcal{B}},{\mathcal{A}}).  
\]
Then there is a canonical injective map
\[
\Phi:
({\mathcal{F}}'(G/Q;\zeta^{\ast}) \otimes \tau)^H
\hookrightarrow
\invHom{G}{{\mathcal{F}}(G/Q;\zeta)}{C(G/H;\tau)}.  
\]
\end{proposition}
\begin{proof}
The proof is essentially the same
 with that of \cite[Lemma 3.2]{xtoshitoshima}
 which treated the case
 where 
$
   ({\mathcal{F}}, {\mathcal{F}}')=({\mathcal{A}}, {\mathcal{B}})
$
and where $Q$ is a minimal parabolic subgroup of $G$.  
For the sake of completeness,
 we repeat the proof 
 with appropriate modifications.

The natural $G$-invariant non-degenerate bilinear form
\[
\langle\, , \, \rangle :
 {\mathcal{F}}(G/Q;\zeta) \times {\mathcal{F}}'(G/Q;\zeta^{\ast})
  \to 
  {\mathbb{C}}
\]
 induces an injective $G$-homomorphism
\[
\Psi: {\mathcal{F}}'(G/Q;\zeta^{\ast}) \hookrightarrow
      \invHom {G}{{\mathcal{F}}(G/Q;\zeta)}{C(G)}
\]
by 
\[
\Psi(\chi)(u)(g):=\langle \pi(g)^{-1}u, \chi\rangle
\quad
\text{for }\,\,
 \chi\in {\mathcal{F}}'(G/Q;\zeta^{\ast})
\,\,{\text{ and }}\,\,
 u \in {\mathcal{F}}(G/Q;\zeta), 
\]
where $\pi$ is the regular representation
 of $G$ on ${\mathcal{F}}(G/Q;\zeta)$.

Taking the tensor product
 with the finite-dimensional representation $\tau$
 followed by collecting $H$-invariant elements,
 we get the linear map $\Phi$
 in Proposition \ref{prop:Poisson-1}.  
\end{proof}

\begin{example}
[Poisson integral transform]
\label{ex:Poisson}
{\rm{
We apply Proposition \ref{prop:Poisson-1}
 in the following setting:
\begin{align*}
&({\mathcal{F}},{\mathcal{F}}')=({\mathcal{B}},{\mathcal{A}}),
\\
&H=K, 
\\
&Q: \text{a minimal parabolic subgroup 
 of $G$}, 
\\
&\tau:\text{ the trivial one-dimensional representation 
 ${\bf{1}}$ of $K$}, 
\\
& \zeta:\text{ a one-dimensional representation 
 of $Q$ 
 such that $\zeta|_{Q \cap K}$ is trivial.  }
\end{align*}
Then ${\mathcal{A}}(G/Q;\zeta^{\ast})$
 is identified with ${\mathcal{A}}(K/(Q \cap K))$
 as a $K$-module,
 and the constant function 
 ${\bf{1}}_K$ on $K/(Q \cap K)$
 gives rise to an element
 of 
$
   ({\mathcal{A}}(G/Q; \zeta^{\ast}) \otimes \tau)^K
$.  
Then ${\mathcal{P}}_{\mu}:=
 \Phi({\bf{1}}_K)$
 in Proposition \ref{prop:Poisson-1}
 coincides with the Poisson integral transform
 for the Riemannian symmetric space
 $G/K$
 (\cite[Chapter 2]{Helgason}):
\[
   {\mathcal{P}}_{\mu}:
   {\mathcal{B}}(G/Q;\zeta) \to C(G/K), 
   \quad
   f \mapsto ({\mathcal{P}}_{\mu}f)(g)= \int_K f(gk) d k.  
\]
See Proposition \ref{prop:Poisson}
 for the preceding results 
 on the image of ${\mathcal{P}}_{\mu}$.  
}}
\end{example}

\begin{proof}
[Proof of Proposition \ref{prop:N}]
The proof is parallel to that of \cite[Theorem 3.17]{xtoshitoshima}.

Let $U_i$ ($i=1,2,\cdots,m$)
 be disjoint $H$-invariant open subsets in $G/Q$.  
We define
\[
\chi_i(g):=
\begin{cases}
1
&\text{if }\,\, g \in U_i, 
\\
0
&\text{if }\,\, g \not\in U_i.  
\end{cases}
\]
Then $\chi_i \in L^2(G/Q) \simeq L^2(K/Q \cap K)$ $(i=1,\cdots,m)$, 
 and they are $H$-invariant 
 and linearly independent.  

We take linearly independent elements $u_1$, $\cdots$, $u_n$
 in $\invHom {H}{\sigma|_H}{\tau}$.  
Taking the dual 
 of the surjective
 $({\mathfrak {g}}, K)$-homomorphism
 $E(G/Q;\lambda(\sigma)) \twoheadrightarrow \sigma$, 
 we have an injective $({\mathfrak{g}},K)$-homomorphism
$\sigma^{\vee}
 \hookrightarrow
 E(G/Q;\lambda(\sigma)^{\ast})
 \subset
 {\mathcal{A}}(G/Q;\lambda(\sigma)^{\ast})$.  
Hence we may regard 
$u_j \in \invHom {H}{\sigma|_H}{\tau} \simeq (\sigma^{\vee} \otimes \tau)^H$
 as $H$-invariant elements
 of ${\mathcal{A}}(G/Q;\lambda(\sigma)^{\ast}) \otimes \tau$.  
Then $\chi_i u_j\in (L^2(G/Q;\lambda(\sigma)^{\ast}) \otimes \tau)^H$
 ($1 \le i \le m$, $1 \le j \le n$)
 are linearly independent.  

Proposition \ref{prop:N} now follows from 
Proposition \ref{prop:Poisson-1}
 with $({\mathcal{F}}, {\mathcal{F}}')=(L^2,L^2)$.  
\end{proof}

\begin{proposition}
\label{prop:L}
Let $Q$ and $Q'$ be parabolic subgroups
 of $G$ and $G'$.  
Suppose that there are $m$ disjoint
 $Q'$-invariant open sets
 in $G/Q$.  
Then 
\[
 \dim \invHom{G'}{L^2(G/Q, \Omega_{G/Q})}{L^2(G'/Q')} \ge m.  
\]
\end{proposition}
\begin{proof}
We apply Proposition \ref{prop:N}
 to $(G \times G'$, 
 $\operatorname{diag} G'$, 
 ${\bf{1}}$, ${\bf{1}}$, 
 $Q \times Q'$)
 for $(G$, $H$, $\sigma$, $\tau$, $Q$).  
Then we have 
\[
    \dim \invHom{G\times G'}{\pi \boxtimes \tau}{C(G \times G'/\operatorname{diag}G')} \ge m, 
\]
where $\pi$ is the Hilbert representation of $G$
 on $L^2(G/Q, \Omega_{G/Q})$
 and $\tau$ is that of $G'$
 on $L^2(G'/Q', \Omega_{G'/Q'})$.

By Proposition \ref{prop:Frob}, 
 we have
\[
    \dim \invHom{G'}{(\pi \boxtimes \tau)|_{\operatorname{diag}G'}}
                    {{\mathbb{C}}} \ge m.   
\]
By Lemma \ref{lem:rest} (2), 
 we get the required lower bound.  
\end{proof}

\subsection{Realization of small representations}
\label{subsec:Qseries}
We end this section
 with a refinement of \cite[Theorem A (2)]{xtoshitoshima}
 which was formulated originally 
 in the category
 of $({\mathfrak {g}}, K)$-modules
 and was proved
 when $Q$ is a minimal parabolic subgroup of $G$.

\begin{definition}
\label{def:Pseries}
{\rm{
Let $Q$ be a parabolic subgroup 
 of a real reductive Lie group $G$.  
Let $\pi$ be an irreducible admissible representation
 of $G$, 
and $\pi_K$ the underlying $({\mathfrak {g}}, K)$-module.  
We say $\pi$ (or $\pi_K$)
 belongs to {\it{$Q$-series}}
 if $\pi_K$ occurs 
 as a subquotient of the induced $({\mathfrak {g}}, K)$-module
 $E(G/Q;\tau)$
 for some finite-dimensional representation $\tau$ of $Q$.  
}}
\end{definition}
By Harish-Chandra's subquotient theorem \cite{HarishChandra54}, 
 all irreducible admissible representations
 of $G$
 belong to $P$-series
 where $P$ is a minimal parabolic subgroup of $G$.  
Loosely speaking,
 the larger a parabolic subgroup $Q$ is, 
the \lq\lq{smaller}\rq\rq\ a representation
 belonging to $Q$-series
 becomes,
 as the following lemma indicates:
\begin{lemma}
\label{lem:GK}
If $\pi_K$ belongs to $Q$-series,
 then its Gelfand--Kirillov dimension, 
 to be denoted by $\operatorname{DIM}(\pi_K)$, 
 satisfies
\[
   \operatorname{DIM}(\pi_K) \le \dim G/Q.  
\]
\end{lemma}
The following result formulates
 that if a subgroup $H$ is 
 \lq\lq{small enough}\rq\rq\
 then the space $(\pi^{-\infty})^H$
 of $H$-invariant distribution vectors of $\pi$
 can be of infinite dimension
 even for a \lq\lq{small}\rq\rq\ admissible representations
 $\pi$:
\begin{corollary}
\label{cor:6.8}
Let $H$ be an algebraic subgroup of $G$, 
 and $Q$ a parabolic subgroup of $G$.  
Assume that $H$ does not have an open orbit in $G/Q$.  
Then for any algebraic finite-dimensional representation 
 $\tau$ of $H$, 
there exists an irreducible admissible Hilbert representation
 $\pi$ of $G$
 such that $\pi$ satisfies 
 the following two properties:
\begin{enumerate}
\item[$\bullet$]
$\pi$ belongs to $Q$-series, 
\item[$\bullet$]
$\dim \invHom {G}{\pi}{C(G/H;\tau)} =\infty$.  
\end{enumerate}
In particular,
$\dim \invHom {G}{\pi^{\infty}}{C^{\infty}(G/H;\tau)}
 =
 \dim \invHom {{\mathfrak {g}}, K}
 {\pi_K}{{\mathcal{A}}(G/H;\tau)}
 =
 \infty$.  
\end{corollary}
\begin{proof}
There exist infinitely many disjoint 
 $H$-invariant open sets
 in $G/Q$
 if $H$ does not have an open orbit in $G/Q$
 (see \cite[Lemma 3.5]{xtoshitoshima}).  
Hence Corollary \ref{cor:6.8}
 follows from Proposition \ref{prop:N}
 because there exist at most finitely many 
 irreducible subquotients 
 in the Hilbert representation
 of $G$
 on $L^2(G/Q, \Omega_{G/Q})$.  
\end{proof}

\begin{corollary}
\label{cor:6.9}
Let $G \supset G'$ be algebraic real reductive 
 Lie groups
 and $Q$ and $Q'$ parabolic subgroups
 of $G$ and $G'$, 
respectively. 
Assume that $Q'$ does not have 
 an open orbit in $G/Q$.  
Then there exist irreducible admissible Hilbert representations
 $\pi$ and $\tau$
 of $G$ and $G'$, 
respectively, 
such that $(\pi,\tau)$ satisfies
 the following two properties:
\begin{enumerate}
\item[$\bullet$]
$\pi$ belongs to $Q$-series,
 $\tau$ belongs to $Q'$-series.  
\item[$\bullet$]
$\dim \invHom{G'}{\pi|_{G'}}{\tau}=\infty$.  
\end{enumerate}
In particular,
 $\dim \invHom{G'}{\pi^{\infty}|_{G'}}{\tau^{\infty}}=\infty$.  
\end{corollary}
\begin{proof}
Corollary \ref{cor:6.9}
 follows from Proposition \ref{prop:L}.  
Since the argument is similar to the proof
 of Corollary \ref{cor:6.8}, 
 and we omit it.  
\end{proof}
\section{Symmetry breaking operators
 and construction of Shintani functions}
\label{sec:break}
In this section
 we construct Shintani functions
 of moderate growth from 
symmetry breaking operators
 of the restriction
 of admissible smooth representations.

\begin{proposition}
\label{prop:break}
Let $\pi^{\infty}$ be a spherical, 
 admissible smooth representation
 of $G$, 
 and $\tau^{\infty}$ that of $G'$.  
Suppose that $\pi^{\infty}$ and $\tau^{\infty}$ have 
 $\ZG G$ and $\ZG {G'}$-infinitesimal characters
 $\lambda$ and $-\nu$, 
 respectively.  
\par\noindent
{\rm{1)}}\enspace
Let ${\bf{1}}_{\pi}$ and ${\bf{1}}_{\tau^{\vee}}$
 be non-zero spherical vectors
 of $\pi_K$ and $\tau_{K'}^{\vee}$, 
 respectively.  
Then there is a natural linear map
\begin{equation}
\label{eqn:break}
\invHom{G'}{\pi^{\infty}}{\tau^{\infty}}
\to 
\operatorname{Sh}_{\operatorname{mod}}(\lambda,\nu), 
\qquad
T \mapsto F
\end{equation}
defined by 
\[
F(g):=\langle T \circ \pi^{\infty}(g) {\bf{1}}_{\pi}, {\bf{1}}_{\tau^{\vee}}
\rangle 
\quad
\text{for }\,\, g \in G.  
\]
\par\noindent
{\rm{2)}}\enspace
Assume that the spherical vectors ${\bf{1}}_{\pi}$
 and ${\bf{1}}_{\tau^{\vee}}$
 are cyclic 
 in $\pi_K$ and $\tau_{K'}^{\vee}$, 
respectively.  
Then \eqref{eqn:break} is injective.  
In particular,
 if both $\pi^{\infty}$ and $\tau^{\infty}$ 
 are irreducible,
 \eqref{eqn:break} is injective.  
\end{proposition}

\begin{remark}
{\rm{
In the setting of Proposition \ref{prop:break}, 
 if we drop the assumption
 that ${\bf{1}}_{\pi}$ is cyclic, 
then the homomorphism \eqref{eqn:break}
 may not be injective.  
In fact,
 we shall see in Section \ref{sec:sbon}
 that there is a countable set
 of $(\lambda,\nu)$
 for which the following three conditions are satisfied:
\begin{enumerate}
\item[$\bullet$]
$\dim_{\mathbb{C}}\invHom{G'}{\pi^{\infty}}{\tau^{\infty}}=2$, 
\item[$\bullet$]
$\dim_{\mathbb{C}}\operatorname{Sh}_{\operatorname{mod}}(\lambda,\nu)=1$, 
\item[$\bullet$]
${\bf{1}}_{\tau^{\vee}}$ is cyclic in $\tau^{\vee}$.  
\end{enumerate}
}}
\end{remark}
\begin{proof}
[Proof of Proposition \ref{prop:break}]
1)\enspace
Since $T \in \invHom {G'}{\pi^{\infty}}{\tau^{\infty}}$, 
 the function $F \in C^{\infty}(G)$ satisfies
\begin{align}
F(hg)=& \langle \tau^{\infty}(h) \circ \tau \circ \pi^{\infty} (g) 
                {\bf{1}}_{\pi}, 
                {\bf{1}}_{\tau^{\vee}}\rangle 
\notag
\\
     =& \langle T \circ \pi^{\infty} (g) {\bf{1}}_{\pi}, 
                ({\tau^{\vee}})^{\infty}(h^{-1}) {\bf{1}}_{\tau^{\vee}}\rangle 
\label{eqn:Fgh}
\end{align}
 for all $h \in G'$ and $g \in G$.  
Therefore we have
\begin{align*}
F(k' g k) =&F(g)
\quad
\text{for }\,\, k ' \in K' \,\,\text{ and }\,\, k \in K, 
\\
(L_Y F)(g)=&\langle T \circ \pi^{\infty}(g) {\bf{1}}_{\pi}, 
                  d \tau^{\vee} (Y) {\bf{1}}_{\tau^{\vee}} \rangle 
\quad
\text{for }\,\, Y \in {\mathfrak {g}}' \subset U({\mathfrak {g}}_{\mathbb{C}}'), 
\\
(R_X F)(g)=&\langle T \circ \pi^{\infty}(g) d \pi (X) {\bf{1}}_{\pi}, 
                   {\bf{1}}_{\tau^{\vee}} \rangle 
\quad
\text{for }\,\, X \in {\mathfrak {g}} \subset U({\mathfrak {g}}_{\mathbb{C}}).  
\end{align*}
Since $u \in \ZG G$ acts on $\pi^{\infty}$
 as the scalar multiple
 of $\chi_{\lambda}(u)$, 
 we have $d\pi^{\infty}(u) {\bf{1}}_{\pi}=\chi_{\lambda}(u){\bf{1}}_{\pi}$, 
 and therefore
 $R_u F =\chi_{\lambda}(u) F$.  
Likewise,
 for $v \in \ZG {G'}$, 
 we have $d (\tau^{\vee})^{\infty}(v) {\bf{1}}_{\tau^{\vee}}
 =\chi_{\nu}(v){\bf{1}}_{\tau^{\vee}}$, 
 and thus $L_v F =\chi_{\nu}(v) F$.  
Hence $F \in \operatorname{Sh}(\lambda,\nu)$.

Let $V_{\pi}^{\infty}$ and $W_{\tau}^{\infty}$
 be the representation spaces 
 of ${\pi}^{\infty}$ and ${\tau}^{\infty}$, 
respectively.  
First we find a continuous seminorm $| \cdot |_1$ on $W_{\tau}^{\infty}$
 and a constant $C_1$
 such that 
\[
   |\langle w, {\bf{1}}_{\tau^{\vee}}\rangle| \le C_1 |w|_1
\quad
  \text{ for any }\,\,
  w \in W_{\tau}^{\infty}.  
\]
Second, 
since $T:V_{\pi}^{\infty} \to W_{\tau}^{\infty}$
 is continuous, 
 there exist a continuous seminorm $| \cdot |_2$ on $V_{\pi}^{\infty}$
 and a constant $C_2$
 such that 
\[
   |T v |_1 \le C_2 |v|_2
\quad
  \text{ for any }\,\,
  v\in V_{\pi}^{\infty}.  
\]
Third, 
since $\pi^{\infty}$ has moderate growth,
 there exist constants $C_3>0$, $d \in {\mathbb{R}}$
 and a continuous seminorm 
 $| \cdot |_3$ on $V_{\pi}^{\infty}$
 such that 
\[
  |{\pi}^{\infty}(g) d {\pi}(u) {\bf{1}}_{\pi}|_2  
  \le C_3  |d {\pi}^{\infty} (u) {\bf{1}}_{\pi}|_3 \| g \|^d
  \quad
  \text{ for any $g \in G$ and for any $u
         \in U({\mathfrak {g}}_{\mathbb{C}})$}.  
\]
Therefore
$
   (R_uF)(g) = \langle T \circ \pi^{\infty}(g) d \pi (u) {\bf{1}}_{\pi}, 
                   {\bf{1}}_{\tau^{\vee}} \rangle 
$
satisfies
the following inequality:
\[
  |(R_u F)(g)| 
  \le 
  C_1 C_2 C_3  |d {\pi}^{\infty} (u) {\bf{1}}_{\pi}|_3 \| g \|^d
  \quad
  \text{ for any }\,\, g \in G.  
\]
Hence $F \in C^{\infty}(G)$ has moderate growth.  

\par\noindent
2)\enspace
Suppose $F\equiv 0$.  
Since ${\bf{1}}_{\tau^{\vee}}$ is a cyclic vector,
 we have $T \circ \pi^{\infty}(g){\bf{1}}_{\pi}=0$
 for any $g \in G$
 by \eqref{eqn:Fgh}.  
Since ${\bf{1}}_{\pi}$ is a cyclic vector,
 we have $T=0$.  
Therefore the map \eqref{eqn:break} is injective.   
\end{proof}

\begin{corollary}
\label{cor:Shtriv}
Suppose that there are $m$ disjoint $P'$-invariant open
 sets 
 in $G/P$.  
Then 
\begin{equation}
\label{eqn:Shtriv}
  \dim_{\mathbb{C}} \operatorname{Sh}_{\operatorname{mod}}
  (\rho_{\mathfrak {g}}, \rho_{\mathfrak {g}'})\ge m.  
\end{equation}
In particular,
 if $\operatorname{Sh}_{\operatorname{mod}}
  (\rho_{\mathfrak {g}}, \rho_{\mathfrak {g}'})$
 is finite-dimensional, 
 then the pair $(G,G')$
 of reductive groups satisfies {\rm{(PP)}}.  
\end{corollary}

\begin{proof}
We denote by $\pi$ the Hilbert representation of $G$
 on $L^2(G/P, \Omega_{G/P})$, 
 and by $\tau$ that of $G'$
 on $L^2(G'/P')$.  
By Proposition \ref{prop:L}, 
 we have
\[
  \dim_{\mathbb{C}} \invHom{G'}{\pi|_{G'}}{\tau} \ge m.  
\]
On the other hand,
 since both $\pi$ and $\tau^{\vee}$
 contain spherical cyclic vectors,
 we have
\[
\dim_{\mathbb{C}} \operatorname{Sh}_{\operatorname{mod}}
     (\rho_{\mathfrak {g}}, -\rho_{\mathfrak {g}'})
\ge 
\dim_{\mathbb{C}} \invHom{G'}{\pi^{\infty}|_{G'}}{\tau^{\infty}}
\]
{}from Proposition \ref{prop:break}.  
Combining these inequalities with \eqref{eqn:Hsmooth}, 
 we have obtained
\[
\dim_{\mathbb{C}} \operatorname{Sh}_{\operatorname{mod}}
     (\rho_{\mathfrak {g}}, -\rho_{\mathfrak {g}'})
\ge m.  
\]
Since $-\rho_{\mathfrak {g}'}$ is conjugate
 to $\rho_{\mathfrak {g}'}$
 by the longest element
 of the Weyl group $W({\mathfrak {j}}_{\mathbb{C}}')$, 
 we have proved \eqref{eqn:Shtriv}.  

Finally,
 if $(G,G')$ does not satisfy (PP), 
 then there exist infinitely many disjoint
 $P'$-invariant open sets in $G/P$, 
 and therefore we get 
$
  \dim_{\mathbb{C}}
  \operatorname{Sh}_{\operatorname{mod}}
  (\rho_{{\mathfrak {g}}}, \rho_{{\mathfrak {g}}'})
  = \infty
$ from \eqref{eqn:Shtriv}.  
\end{proof}

\section{Boundary values of Shintani functions}
\label{sec:PS}

In this section we realize Shintani functions
 as joint eigenfunctions
 of invariant differential operators
 on the Riemannian symmetric space
 $X=(G \times G')/(K \times K')$, 
and then as hyperfunctions
 on the minimal boundary 
 $Y=(G \times G')/(P \times P')$
 of the compactification of $X$. 
The main results
 of this section 
 are Theorems \ref{thm:nonzero} and \ref{thm:ps}.  
We prove these theorems
 in Sections \ref{subsec:nz}
 and \ref{subsec:ps}, 
 respectively,
 after giving a brief summary
 of the preceding results 
 of harmonic analysis 
 on Riemannian symmetric spaces
 in Sections \ref{subsec:DGK} and \ref{subsec:Pb}.  

\subsection{Symmetry breaking of principal series representations}
\label{subsec:PS}
Denote by $\theta$ the Cartan involution 
of the Lie algebra ${\mathfrak {g}}$
 corresponding to the maximal compact subgroup $K$
 of $G$.  
We take a maximal abelian subspace ${\mathfrak {a}}$
 in the vector space $\{X \in {\mathfrak {g}}: \theta X = -X\}$, 
 and set 
\[
   W({\mathfrak {a}}):=N_K({\mathfrak {a}})/Z_K({\mathfrak {a}}).  
\]
We fix a positive system 
 $\Sigma^+({\mathfrak {g}}, {\mathfrak {a}})$
 of the restricted root system 
 $\Sigma({\mathfrak {g}}, {\mathfrak {a}})$, 
 and define a minimal parabolic subalgebra
 ${\mathfrak {p}}$ of ${\mathfrak {g}}$ by 
\[
   {\mathfrak {p}}
   ={\mathfrak {m}}+{\mathfrak {a}}+{\mathfrak {n}}
   ={\mathfrak {l}}+{\mathfrak {n}}, 
\]
 where 
$
   {\mathfrak {l}}:=Z_{\mathfrak {g}}({\mathfrak {a}})
  =\{X \in {\mathfrak {a}}
    :[H,X]=0\,\,\text{ for all }\,\,H \in{\mathfrak {a}} \}
$, 
 ${\mathfrak {m}}:={\mathfrak {l}} \cap {\mathfrak {k}}$, 
 and ${\mathfrak {n}}$ is the sum
 of the root spaces
 for all $\alpha \in \Sigma^+({\mathfrak {g}}, {\mathfrak {a}})$.  
Let $P=MAN$ be the minimal parabolic subgroup
 of $G$
 with Lie algebra ${\mathfrak {p}}$.

We take a Cartan subalgebra ${\mathfrak {t}}$
 in ${\mathfrak {m}}$.  
Then ${\mathfrak {j}}:={\mathfrak {t}}+{\mathfrak {a}}$
 is a maximally split Cartan subalgebra
 of ${\mathfrak {g}}$.  
We fix a positive system
 $\Delta^+({\mathfrak {m}}_{\mathbb{C}},{\mathfrak {t}}_{\mathbb{C}})$.  
Let $\rho_{{\mathfrak {n}}} \in {\mathfrak {a}}^{\vee}$
 be half the sum 
of the elements in $\Sigma^+({\mathfrak {g}}, {\mathfrak {a}})$
 counted with multiplicities,
 and $\rho_{{\mathfrak {l}}} \in {\mathfrak {t}}_{\mathbb{C}}^{\vee}$
 that of $\Delta^+({\mathfrak {m}}_{\mathbb{C}},{\mathfrak {t}}_{\mathbb{C}})$.
The positive systems
 $\Sigma^+({\mathfrak {g}}, {\mathfrak {a}})$
 and $\Delta^+({\mathfrak {m}}_{\mathbb{C}},{\mathfrak {t}}_{\mathbb{C}})$
 determine naturally 
 a positive system 
 $\Delta^+({\mathfrak {g}}_{\mathbb{C}},{\mathfrak {j}}_{\mathbb{C}})$.  
Then we have 
\[
   \rho_{{\mathfrak {g}}}=\rho_{{\mathfrak {l}}}+\rho_{{\mathfrak {n}}}
 \in {{\mathfrak {j}}}_{\mathbb{C}}^{\vee}
 ={{\mathfrak {t}}}_{\mathbb{C}}^{\vee}+{{\mathfrak {a}}}_{\mathbb{C}}^{\vee}, 
\]  
where we regard
 ${{\mathfrak {t}}}_{\mathbb{C}}^{\vee}$
 and ${{\mathfrak {a}}}_{\mathbb{C}}^{\vee}$
 as subspaces of ${{\mathfrak {j}}}_{\mathbb{C}}^{\vee}$
  via the direct sum decomposition 
 ${\mathfrak {j}}={\mathfrak {t}}+{\mathfrak {a}}$.  
Then $\rho_{{\mathfrak {l}}}+{\mathfrak {a}}_{\mathbb{C}}^{\vee}
 =\rho_{{\mathfrak {g}}}+ {\mathfrak {a}}_{\mathbb{C}}^{\vee}$
 is an affine subspace of 
 ${\mathfrak {j}}_{\mathbb{C}}^{\vee}$.

Analogous notation is applied
 to the reductive subgroup $G'$.  
In particular,
 ${\mathfrak {j}}'={\mathfrak {t}}'+{\mathfrak {a}}'$
 is a maximally split Cartan subalgebra
 of ${\mathfrak {g}}'$.

We recall 
$(\lambda,\nu) \in 
 {\mathfrak {j}}_{\mathbb{C}}^{\vee}/W({\mathfrak {j}}_{\mathbb{C}})
 \times
 ({\mathfrak {j}}_{\mathbb{C}}')^{\vee}/W({\mathfrak {j}}_{\mathbb{C}}')
 \simeq
 \HomCalg {\ZG G \times \ZG {G'}}{{\mathbb{C}}}$. 
Let us begin with a non-vanishing condition
 for $\operatorname{Sh} (\lambda,\nu)$.  

\begin{theorem}
\label{thm:nonzero}
If $\operatorname{Sh}(\lambda,\nu)\ne \{0\}$, 
then 
\begin{equation}
\label{eqn:ln}
  \lambda \in 
  W({\mathfrak {j}}_{\mathbb{C}}) 
 (\rho_{{\mathfrak {l}}} + {\mathfrak {a}}_{\mathbb{C}}^{\vee})
 \,\,\text{ and }\,\,
  \nu \in 
  W({\mathfrak {j}}_{\mathbb{C}}') 
 (\rho_{{\mathfrak {l}}'} + ({\mathfrak {a}}_{\mathbb{C}}')^{\vee}).  
\end{equation}
\end{theorem}

We shall give a proof of Theorem \ref{thm:nonzero}
 in Section \ref{subsec:nz}.  
\vskip 1pc

Next we consider a construction
of Shintani functions
 under the assumption \eqref{eqn:ln}.  
Suppose $\lambda \in {\mathfrak {j}}_{\mathbb{C}}^{\vee}$
 satisfies $\lambda - \rho_{{\mathfrak {l}}} \in {\mathfrak {a}}_{\mathbb{C}}
^{\vee}$.  
Then there exists $\lambda_+ \in {\mathfrak {a}}_{\mathbb{C}}^{\vee}$
 such that $\lambda_+$
 satisfies the following two conditions:
\begin{align}
  &\lambda_+ - \rho_{\mathfrak {n}}
   = w (\lambda - \rho_{\mathfrak {l}})
\,\,\text{ for some }\,\, w \in W({\mathfrak {a}}).  
\label{eqn:lmd+a}
\\
  &\operatorname{Re}
   \langle \lambda_+ - \rho_{\mathfrak {n}}, \alpha \rangle
   \ge 0
\,\,\text{ for any } \alpha \in \Sigma^+({\mathfrak {g}}, {\mathfrak {a}}).  
\label{eqn:lmd+}
\end{align}
Similarly,
 suppose $\nu \in ({\mathfrak {j}}_{\mathbb{C}}')^{\vee}$
 satisfies 
 $\nu - \rho_{{\mathfrak {l}}'} \in ({\mathfrak {a}}_{\mathbb{C}}')
^{\vee}$.  
Then there exists $\nu_- \in ({\mathfrak {a}}_{\mathbb{C}}')^{\vee}$
 satisfying the following two conditions:
\begin{align}
  &\nu_- - \rho_{\mathfrak {n}'}
   = w' (-\nu + \rho_{\mathfrak {l}'})
\,\,\text{ for some }\,\, w' \in W({\mathfrak {a}}').  
\notag
\\
  &\operatorname{Re}
   \langle \nu_- - \rho_{\mathfrak {n}'}, \alpha \rangle
   \le 0
\,\,\text{ for any } \alpha \in \Sigma^+({\mathfrak {g}}', {\mathfrak {a}}').  
\label{eqn:nu-}
\end{align}

\begin{theorem}
\label{thm:ps}
Suppose that 
$
    \lambda \in {\mathfrak {j}}_{\mathbb{C}}^{\vee}
$
 and 
$
   \nu \in ({\mathfrak {j}}_{\mathbb{C}}')^{\vee}
$
 satisfy 
$
   \lambda + \rho_{{\mathfrak {l}}} 
   \in {\mathfrak {a}}_{\mathbb{C}}^{\vee}
$
 and 
$
   \nu+ \rho_{{\mathfrak {l}}'} 
\in ({\mathfrak {a}}_{\mathbb{C}}')^{\vee}
$.  
Let $\lambda_+$ and $\nu_-$
 be defined as above.  
\begin{enumerate}
\item[{\rm{1)}}]
There is a natural injective linear map
\begin{equation}
\label{eqn:ps}
\invHom{G'}{C^{\infty}(G/P;\lambda_+)}{C^{\infty}(G'/P';\nu_-)}
 \hookrightarrow
 \operatorname{Sh}_{\operatorname{mod}}(\lambda,\nu).  
\end{equation}
\item[{\rm{2)}}]
If $G$, $G'$ are classical groups,
 then \eqref{eqn:ps} is a bijection:
\begin{equation}
\label{eqn:bs}
\invHom{G'}{C^{\infty}(G/P;\lambda_+)}{C^{\infty}(G'/P';\nu_-)}
 \overset {\sim} \to 
 \operatorname{Sh}_{\operatorname{mod}}(\lambda,\nu).  
\end{equation}
\end{enumerate}
\end{theorem}

We shall prove Theorem \ref{thm:ps}
 in Section \ref{subsec:ps}.  

\begin{remark}
\label{rem:ZGgen}
{\rm{
As the proof shows,
 the bijection \eqref{eqn:bs} holds
 for generic $(\lambda,\nu)$
 even when $G$ or $G'$
 are exceptional groups.  
}}
\end{remark}
%%%%%%%%%%%%%%%%%%%%%%%%%%%%%%%%%%%%%%%%%%%%%%%%%%%%
\subsection{Invariant differential operators}
\label{subsec:DGK}
In this and next subsections,
 we give a quick review
 of the preceding results
 of harmonic analysis 
 on Riemannian symmetric spaces.  
We denote by ${\mathbb{D}}(G/K)$
 the ${\mathbb{C}}$-algebra
 consisting of all $G$-invariant differential operators
 on the Riemannian symmetric space
 $G/K$.  
It is isomorphic
 to a polynomial ring 
 of $(\dim_{\mathbb{R}}{\mathfrak {a}})$-generators.  
More precisely,
 let $\gamma':U({\mathfrak {g}}_{\mathbb{C}})
 \to U({\mathfrak {a}}_{\mathbb{C}})
 = S({\mathfrak {a}}_{\mathbb{C}})$
 be the projection
 to the second factor
 of the decomposition 
$
   U({\mathfrak {g}}_{\mathbb{C}})
   =
   ({\mathfrak {k}}_{\mathbb{C}} U({\mathfrak {g}}_{\mathbb{C}})
    + 
    U({\mathfrak {g}}_{\mathbb{C}}) {\mathfrak {n}}_{\mathbb{C}})
    \oplus U({\mathfrak {a}}_{\mathbb{C}}).  
$
Then we have the Harish-Chandra isomorphism
\begin{equation}
\label{eqn:HCiso}
   {\mathbb{D}}(G/K) 
\underset R{\overset \sim \leftarrow}
    U({\mathfrak {g}}_{\mathbb{C}})^K
   / 
    U({\mathfrak {g}}_{\mathbb{C}})^K
    \cap 
    U({\mathfrak {g}}_{\mathbb{C}}){\mathfrak {k}}_{\mathbb{C}}
\underset \gamma{\overset \sim \to} 
   S({\mathfrak {a}}_{\mathbb{C}})^{W({\mathfrak {a}})},
\end{equation}
where  
$\gamma:U({\mathfrak {g}}_{\mathbb{C}})
 \to S({\mathfrak {a}}_{\mathbb{C}})$
 is defined by 
\[
  \langle \gamma(u), \lambda \rangle
  =
  \langle \gamma'(u), \lambda - \rho_{{\mathfrak {n}}}\rangle
\qquad
 \text{ for all $\lambda \in {\mathfrak {a}}_{\mathbb{C}}^{\vee}$, }
\]
which is a generalization of \eqref{eqn:HCisom}, 
 see \cite[Chapter II]{Helgason}.  
Through \eqref{eqn:HCiso}, 
 we have a bijection
\begin{equation}
\label{eqn:ZGK}
  \HomCalg{{\mathbb{D}}(G/K)}{{\mathbb{C}}}
  \simeq
  {\mathfrak {a}}_{\mathbb{C}}^{\vee}/W({\mathfrak {a}}), 
  \qquad
  \psi_{\mu} \leftrightarrow \mu, 
\end{equation}
given by 
$
  \psi_{\mu}(R_v)
  =
  \langle \gamma(v), \mu \rangle
  =
  \langle \gamma'(v), \lambda-\rho_{{\mathfrak {n}}} \rangle
$
 for 
$v \in U({\mathfrak {g}}_{\mathbb{C}})^K$.

Comparing the two bijections
$\HomCalg{\ZG G}{{\mathbb{C}}}
\simeq 
{\mathfrak {j}}_{\mathbb{C}}^{\vee}/W({\mathfrak {j}}_{\mathbb{C}})
$
 (see \eqref{eqn:ZG})
 and 
 \eqref{eqn:ZGK}
via the ${\mathbb{C}}$-algebra homomorphism
\begin{equation}
\label{eqn:ZD}
   \ZG G \subset U({\mathfrak {g}}_{\mathbb{C}})^K
         \overset R \to {\mathbb{D}}(G/K),
\end{equation}
 we have
\begin{equation}
\label{eqn:ZDGK}
\psi_{\mu} \circ R = \chi_{\mu + \rho_{{\mathfrak {n}}}}
\quad
\text{ on $\ZG G$
 for all $\mu \in {\mathfrak {a}}_{\mathbb{C}}^{\vee}$}.  
\end{equation}
By \eqref{eqn:ZDGK}, 
 we have 
\begin{equation}
\label{eqn:GKZG}
C^{\infty}(G/K;{\mathcal{M}}_{\mu})
\subset 
C^{\infty}(G/K;\chi_{\mu + \rho_{\mathfrak {l}}}^R).  
\end{equation}
For a simple Lie group $G$, 
 it is known \cite{Helgason92}
 that the ${\mathbb{C}}$-algebra homomorphism
 \eqref{eqn:ZD} is surjective 
 if and only if $({\mathfrak {g}}, {\mathfrak {k}})$
 is not one of the following pairs:
\begin{equation*}
({\mathfrak {e}}_{6(-14)}, {\mathfrak {so}}(10)+{\mathbb{R}}), \,
({\mathfrak {e}}_{6(-26)}, {\mathfrak {f}}_{4(-52)}), \,
({\mathfrak {e}}_{7(-25)}, {\mathfrak {e}}_{6(-78)}+{\mathbb{R}}),\, 
({\mathfrak {e}}_{8(-24)}, 
{\mathfrak {e}}_{7(-133)}+{\mathfrak {su}}(2)).
\end{equation*}

For ${\mathcal{F}}={\mathcal{A}}, {\mathcal{B}}$, 
 $C^{\infty}$, 
 or ${\mathcal{D}}'$, 
 we denote by ${\mathcal{F}}(G/K;{\mathcal{M}}_{\mu})$
 the space of all $F \in {\mathcal{F}}(G/K)$
 such that $F$ satisfies
 the system of the following partial
 differential equations:
\begin{equation}
   D F = \psi_{\mu} (D) F 
  \quad
  \text{ for all }\,\, D \in {\mathbb{D}}(G/K).  
\tag{${\mathcal{M}}_{\mu}$}
\end{equation}

Since the Laplacian $\Delta$ 
 on the Riemannian symmetric space $G/K$
 is an elliptic differential operator
 and belongs to ${\mathbb{D}}(G/K)$, 
 we have
\[
     {\mathcal{A}}(G/K;{\mathcal{M}}_{\mu})
     =
     {\mathcal{B}}(G/K;{\mathcal{M}}_{\mu})
     =
     C^{\infty}(G/K;{\mathcal{M}}_{\mu})
     =
     {\mathcal{D}}'(G/K;{\mathcal{M}}_{\mu})
\]
 by the elliptic regularity theorem.

\subsection{Poisson transform and boundary maps}
\label{subsec:Pb}
Given $\mu \in {\mathfrak {a}}_{\mathbb{C}}^{\vee}$, 
 we lift and extend it
 to a one-dimensional representation
 of the minimal parabolic subgroup $P=MAN$
 by 
\[
   P \to {\mathbb{C}}^{\times}, 
   \quad
   m \exp H n \mapsto e^{\langle \mu, H\rangle}
\]
 for $m \in M$, 
 $H \in {\mathfrak {a}}$, 
and $n \in N$.

\begin{remark}
\label{rem:shift}
{\rm{
In the field of harmonic analysis 
 on symmetric spaces
 people sometimes adopt the opposite signature
 of the (normalized) parabolic induction 
 which is used in the representation theory
 of real reductive groups.  
Since our definition
of parabolic induction
 does not involve the \lq\lq{$\rho$-shift}\rq\rq\
 ({\it{i.e.}}, 
 unnormalized parabolic induction
 where $\sqrt{-1}{\mathfrak {a}}^{\vee}
 + \rho_{\mathfrak {n}}$ is the unitary axis), 
 the $G$-module $C^{\infty}(G/P;\mu)$ in our notation
 corresponds to $C^{\infty}(G/P;{\mathcal{L}}_{\rho_{\mathfrak {n}}-\mu})$
 in \cite{Helgason,OS1}.  
}}
\end{remark}
With this remark in mind,
 we summarize some known results
 that we need:
\begin{proposition}
\label{prop:Poisson}
\begin{enumerate}
\item[{\rm{1)}}]
The $({\mathfrak {g}}, K)$-module
$E(G/P; \mu)$ has $\ZG G$-infinitesimal character 
 $\mu + \rho_{{\mathfrak {g}}}
  = \mu + \rho_{{\mathfrak {n}}} + \rho_{{\mathfrak {l}}}
  \in {\mathfrak {j}}_{\mathbb{C}}^{\vee}/W({\mathfrak {j}}_{\mathbb{C}})$.  

\item[{\rm{2)}}]
The $({\mathfrak {g}}, K)$-module
 $E(G/P; \mu)$ is spherical 
 for all $\mu \in {\mathfrak {a}}_{\mathbb{C}}^{\vee}$.  
Furthermore,
 the unique {\rm{(}}up to scalar{\rm{)}} non-zero spherical vector
 is cyclic
 if $\mu$ satisfies
\[
 \operatorname{Re} 
 \langle \mu-\rho_{\mathfrak {n}}, \alpha\rangle
  \ge 0
  \,\,
\text{ for any }\,\,
  \alpha \in \Sigma^+({\mathfrak {g}}, {\mathfrak {a}}).  
\]
\item[{\rm{3)}}]
For all $\mu \in {\mathfrak {a}}_{\mathbb{C}}^{\vee}$, 
 the Poisson transform ${\mathcal{P}}_{\mu}$ maps
 into the space of joint eigenfunctions
 of the ${\mathbb{C}}$-algebra ${\mathbb{D}}(G/K)$:
\begin{equation}
\label{eqn:BA}
   {\mathcal{P}}_{\mu}:
   {\mathcal{B}}(G/P;\mu)
   \to 
   {\mathcal{A}}(G/K;{\mathcal{M}}_{\rho_{\mathfrak {n}}-\mu}).  
\end{equation}
\item[{\rm{4)}}]
The Poisson transform \eqref{eqn:BA}
  is bijective
 if $\mu$ satisfies
\begin{equation}
\label{eqn:neg}
 \operatorname{Re} 
 \langle \mu-\rho_{\mathfrak {n}}, \alpha\rangle
  \le 0
  \,\,
\text{ for any }\,\,
  \alpha \in \Sigma^+({\mathfrak {g}}, {\mathfrak {a}}).  
\end{equation}
\item[{\rm{5)}}]
The Poisson transform ${\mathcal{P}}_{\mu}$ induces a bijection
\begin{equation*}
   {\mathcal{P}}_{\mu}:
   {\mathcal{D}}'(G/P;\mu)
   \to 
   {\mathcal{A}}_{\operatorname{mod}}(G/K;{\mathcal{M}}_{\rho_{\mathfrak {n}}-\mu})
\end{equation*}
if \eqref{eqn:neg} is satisfied.  
\end{enumerate}
\end{proposition} 

\begin{proof}
The first statement is elementary.  
See Kostant \cite{Kostant}
for (2), 
 and Helgason \cite{Helgason76}
 for (3).  
The fourth statement was proved 
 in Kashiwara {\it{et.al.}} \cite{KKMOOT}
 by using the theory of regular singularity 
 of a system of partial differential equations.  
For the proof of the fifth statement,
 see Oshima--Sekiguchi \cite{OS1}
 or Wallach \cite[Theorem 11.9.4]{WaI}.   
We note
 that for $f \in {\mathcal{A}}(G/K;{\mathcal{M}}_{\rho_{{\mathfrak {n}}}-\mu})$, 
$f$ has moderate growth 
 (Definition \ref{def:mod})
 if and only if
 $f$ has at most exponential growth
 in the sense that
 there exist constants
 $d \in {\mathbb{R}}$
 and $C >0$
 such that
$
  |f(x)| \le C\| x \|^d
$
 for all 
$
  x \in G.  
$  
\end{proof}

\subsection{Proof of Theorem \ref{thm:nonzero}}
\label{subsec:nz}
In Lemma \ref{lem:ShPP}, 
 we realized the Shintani space $\operatorname{Sh} (\lambda,\nu)$
 in $C^{\infty}((G \times G')/\operatorname{diag}G')$.  
We give another realization
 of $\operatorname{Sh} (\lambda,\nu)$:
\begin{lemma}
\label{lem:ShGK}
The multiplication map
\[
 \psi: G \times G' \to G, 
  \quad
 (g,g')\mapsto (g')^{-1}g
\]
 induces the following bijection:
\begin{equation}
\label{eqn:ShGK}
  \psi^{\ast}:
  \operatorname{Sh} (\lambda,\nu)
  \overset \sim \to 
  C^{\infty}((G \times G')/(K \times K');\chi_{\lambda,\nu}^R)
^{\operatorname{diag}G'}.  
\end{equation}
\end{lemma}
\begin{proof}
We set
$
   C^{\infty}(K' \backslash G/K)
   :=
   \{f \in C^{\infty}(G)
    :
     f(k'gk)=f(g)
     \quad
     \text{for all }\,\,
     k' \in K'
     \,\,\text{and}\,\, k \in K\}.  
$
The pull-back $\psi^{\ast}$ 
 of functions induces
 a bijective linear map
\begin{equation*}
\begin{array}{cccc}
&C^{\infty}(G) &\,\,\overset \sim \to \,\,&C^{\infty}(G \times G')^{\operatorname{diag}G'}
\\
& \cup         &                & \cup
\\
&C^{\infty}(K' \backslash G/K) 
&\,\,\underset{\psi^{\ast}}{\overset \sim \to}\,\, 
&C^{\infty}((G\times G')/(K \times K'))^{\operatorname{diag}G'}.  
\end{array}
\end{equation*}
On the other hand,
 for $X \in {\mathfrak {g}}$
 and $Y \in {\mathfrak {g}}'$, 
 we have
\[
 R_X R_Y (\psi^{\ast} f)
 =
 \psi^{\ast} (R_X L_Y f).  
\]
Thus Lemma \ref{lem:ShGK} is proved.  
\end{proof}

\begin{proof}
[Proof of Theorem \ref{thm:nonzero}]
Suppose $\operatorname{Sh} (\lambda,\nu) \ne \{0\}$.  
Then, 
by Lemma \ref{lem:ShGK}, 
we have
\[
   V_{\lambda,\nu}
   :=
   C^{\infty}((G \times G')/(K \times K');\chi_{\lambda,\nu}^R)
   \ne \{0\}.  
\]
Since $R(\ZG{G \times G'})$ is an ideal
 in ${\mathbb{D}}((G \times G')/(K \times K'))$
 of finite-codimension,
 we can take the boundary values
 of $V_{\lambda,\nu}$
 inductively to the hyperfunction-valued
 principal series representations
 of $G \times G'$
 as in \cite[Section 2]{xtoshitoshima}.  
To be more precise,
 there exist $\mu_1, \cdots, \mu_N
 \in {\mathfrak {a}}_{\mathbb{C}}^{\vee}
 \times ({\mathfrak {a}}_{\mathbb{C}}')^{\vee}$
 and $(G \times G')$-invariant subspaces
\[
  \{0\} =V(0) \subset V(1) \subset \cdots \subset 
  V(N) = V_{\lambda,\nu}
\]
 such that the quotient space
 $V(j)/V(j-1)$ is isomorphic to a subrepresentation
 of the spherical principal series representation
 ${\mathcal{B}}((G \times G')/(P \times P');\mu_j)$
 as $(G \times G')$-modules.  

Comparing the $\ZG{G \times G'}$-infinitesimal characters
 of $V_{\lambda,\nu}$ 
 and ${\mathcal{B}}((G \times G')/(P \times P');\mu_j)$, 
 we get Theorem \ref{thm:nonzero}.  
\end{proof}

\subsection{Proof of Theorem \ref{thm:ps}}
\label{subsec:ps}
\begin{proof}
[Proof of Theorem \ref{thm:ps}]
1)\enspace
Since $\lambda_+$ satisfies \eqref{eqn:lmd+}, 
 the $({\mathfrak {g}}, K)$-module
 $E(G/P;\lambda_+)$ contains 
 a cyclic spherical vector
 by Proposition \ref{prop:Poisson}.  
Similarly,
 the $({\mathfrak {g}}', K')$-module
\[
   E(G'/P';\nu_-)_{K'}^{\vee} \simeq E(G'/P';\nu_-^{\ast})_{K'}
\]
 has a cyclic vector
 because $\nu_-^{\ast}=- \nu_- + 2 \rho_{{\mathfrak {n}}'}$
 satisfies
\[
  \operatorname{Re}
  \langle
  \nu_-^{\ast}-\rho_{{\mathfrak {n}}'}, \alpha
  \rangle
  \ge 0
  \quad
  \text{for any $\alpha \in \Sigma^+({\mathfrak {g}}', {\mathfrak {a}}')$}
\]
 by \eqref{eqn:nu-}.  
Hence the first statement follows from Proposition \ref{prop:break}.  

2)\enspace
In view of the definition
 of moderate growth
 (Definition \ref{def:mod}), 
 we see 
that the bijection $\psi^{\ast}$
 in \eqref{eqn:ShGK}
 induces the following bijection:
\begin{equation}
\label{eqn:ShA}
   \operatorname{Sh}_{\operatorname{mod}}(\lambda,\nu)
  \overset \sim \to 
  C_{\operatorname{mod}}^{\infty}
  ((G \times G')/(K \times K');\chi_{\lambda,\nu}^R)^{\operatorname{diag}G'}.  
\end{equation}

Since the ${\mathbb{C}}$-algebra homomorphism
 $R:\ZG{G \times G'} \to {\mathbb{D}}(G \times G'/K \times K')$
 is surjective 
 for classical groups $G$ and $G'$, 
 the isomorphism \eqref{eqn:ShA}
 implies the following bijection 
\[
   \operatorname{Sh}_{\operatorname{mod}}(\lambda,\nu)
  \simeq 
  C_{\operatorname{mod}}^{\infty}
  ((G \times G')/(K \times K'); {\mathcal{M}}_{(\lambda+\rho_{\mathfrak{l}}, \nu+\rho_{{\mathfrak{l}}'})})^{\operatorname{diag}G'}
\]
by \eqref{eqn:GKZG}.  
In turn,
 combining with the Poisson transform, 
 we have obtained the following natural isomorphism
 by Proposition \ref{prop:Poisson} (5):
\[
   \operatorname{Sh}_{\operatorname{mod}}(\lambda,\nu)
  \simeq 
  {\mathcal{D}}'
  ((G \times G')/(P \times P');
   \lambda_+^{\ast} \boxtimes \nu_-)^{\operatorname{diag}G'}.  
\]

By \cite[Proposition 3.2]{xtsbon},
 we proved the following natural bijection: 
\[
\invHom{G'}{C^{\infty}(G/P;\lambda_+)}{C^{\infty}(G'/P';\nu_-)}
\simeq
{\mathcal{D}}'((G \times G')/(P \times P');
 \lambda_+^{\ast} \boxtimes \nu_-)^{\operatorname{diag}G'}.  
\]
Hence we have completed the proof
 of Theorem \ref{thm:ps}.
\end{proof}

\section{Shintani functions for 
 $(O(n+1,1),O(n,1))$}
\label{sec:sbon}
It has been an open problem
 to find $\dim_{\mathbb{C}}\operatorname{Sh} (\lambda,\nu)$
 in the Archimedean case
 (\cite[Remark 5.6]{murasesugano}).  
In this section,
 by using a classification
of symmetry breaking operators
 between spherical principal series representations
 of the pair $(G,G')=(O(n+1,1),O(n,1))$
 in a recent work 
 \cite{xtsbon}
 with B. Speh, 
 we determine the dimension
 of $\operatorname{Sh} (\lambda,\nu)$
 in this case.

We denote by $[x]$
 the greatest integer
 that does not exceed $x$.  
For the pair
 $(G,G')=(O(n+1,1),O(n,1))$, 
 $(\ZG G,\ZG{G'})$-infinitesimal characters
 $(\lambda,\nu)$ are parametrized
 by 
\[
  {\mathfrak{j}}_{\mathbb{C}}^{\vee}
  /
  W({\mathfrak{j}}_{\mathbb{C}})
  \times
  ({\mathfrak{j}}_{\mathbb{C}}')^{\vee}
  /
  W({\mathfrak{j}}_{\mathbb{C}}')
  \simeq
  {\mathbb{C}}^{[\frac{n+2}{2}]}/W_{[\frac{n+2}{2}]}
  \times 
  {\mathbb{C}}^{[\frac{n+1}{2}]}/W_{[\frac{n+1}{2}]}
\]
 in the standard coordinates,
 where $W_k :={\mathfrak{S}}_k \ltimes ({\mathbb{Z}}/2{\mathbb{Z}})^k$.  

\begin{theorem}
\label{thm:sbon}
Let $(G,G')=(O(n+1,1),O(n,1))$.  
\begin{enumerate}
\item[{\rm{1)}}]
The following three conditions on $(\lambda,\nu)$ are equivalent:
\begin{enumerate}
\item[{\rm{(i)}}]
$\operatorname{Sh} (\lambda,\nu) \ne \{0\}$.  
\item[{\rm{(ii)}}]
$\operatorname{Sh}_{\operatorname{mod}}(\lambda,\nu)\ne \{0\}$.
\item[{\rm{(iii)}}]
In the standard coordinates
\begin{align}
\lambda&=w(\frac{n}{2}+t,\frac{n}{2}-1,\frac{n}{2}-2,\cdots,\frac{n}{2}-[\frac{n}{2}]), 
\label{eqn:solmd}
\\
\nu&=w'(\frac{n-1}{2}+s,\frac{n-1}{2}-1,\cdots,\frac{n-1}{2}-[\frac{n-1}{2}]), 
\notag
\end{align}
for some $t$, $s$ $\in {\mathbb{C}}$,
 $w \in W_{[\frac {n+2} 2]}$, 
 and $w' \in W_{[\frac {n+1} 2]}$.  
\end{enumerate}
\item[{\rm{2)}}]
If $(\lambda,\nu)$ satisfies {\rm{(iii)}} in {\rm{(1)}}, 
then 
\[
\dim_{\mathbb{C}} \operatorname{Sh}_{\operatorname{mod}}(\lambda,\nu)
=1.  
\]
\end{enumerate}
\end{theorem}

\begin{proof}
It is sufficient 
 to prove the implication
 (i) $\Rightarrow$ (iii) and (2).

For $\lambda 
\in {\mathfrak {j}}_{\mathbb{C}}^{\vee}
/W({\mathfrak {j}}_{\mathbb{C}})
\simeq {\mathbb{C}}^{[\frac{n+1}{2}]}/W_{[\frac{n+1}{2}]}
$, 
 $\lambda$ belongs to $\rho_{{\mathfrak {l}}}+{\mathfrak {a}}_{\mathbb{C}}^{\vee} \mod W({\mathfrak {j}}_{\mathbb{C}})$
 if and only if 
 $\lambda$ is of the form 
 \eqref{eqn:solmd}
 for some $t \in {\mathbb{C}}$
 and $w \in W_{[\frac{n+2}{2}]}$.  
Similarly for $\nu 
\in ({\mathfrak {j}}_{\mathbb{C}}')^{\vee}
/W({\mathfrak {j}}_{\mathbb{C}}')$.  
Hence the implication (i) $\Rightarrow$ (iii) holds
 as a special case of Theorem \ref{thm:nonzero}.

Next, 
 suppose that $(\lambda,\nu)$ satisfies (iii).  
Without loss of generality,
 we may and do assume
 $\operatorname{Re} t \ge \frac n 2$
 and  
 $\operatorname{Re} s \le \frac {n-1} 2$.  
In this case the unique element $\lambda_+ \in {\mathfrak {a}}_{\mathbb{C}}$
 satisfying 
 \eqref{eqn:lmd+a} and \eqref{eqn:lmd+}
 is equal to $t$
 if we identify 
 ${\mathfrak {a}}_{\mathbb{C}}^{\vee}$
 with ${\mathbb{C}}$
 via the standard basis $\{e_1\}$
 of ${\mathfrak {a}}^{\vee}$
 such that 
 $\Sigma({\mathfrak {g}}, {\mathfrak {a}})
 =\{\pm e_1\}$.  
Similarly,
 $\nu_-=s$ via $({\mathfrak {a}}_{\mathbb{C}}')^{\vee}
 \simeq {\mathbb{C}}$.

We define a discrete subset
 of ${\mathfrak {a}}_{\mathbb{C}}^{\vee}
 \oplus ({\mathfrak {a}}_{\mathbb{C}}')^{\vee}
 \simeq {\mathbb{C}}^2$
 by 
\[
L_{\operatorname{even}}
  :=\{(a,b) \in {\mathbb{Z}}^2:
     a \le b \le 0,
     \,\,
     a \equiv b \mod 2\}.  
\]

According to \cite[Theorem 1.1]{xtsbon}, 
 we have
\[
\dim_{\mathbb{C}}
\invHom{G'}{C^{\infty}(G/P;a)}{C^{\infty}(G'/P';b)}
\simeq
\begin{cases}
1 
& \text{ if }\,\, (a,b) \in {\mathbb{C}}^2\setminus L_{\operatorname{even}}, 
\\
2 &
\text{ if }\,\, (a,b) \in L_{\operatorname{even}}.  
\end{cases}
\]

Since $(\lambda_+,\nu_-)
=(t,s) \not\in L_{\operatorname{even}}$, 
 we conclude that 
\[
\dim_{\mathbb{C}}\operatorname{Sh}_{\operatorname{mod}}(\lambda,\nu)
=
\dim_{\mathbb{C}}
\invHom{G'}{C^{\infty}(G/P;\lambda_+)}{C^{\infty}(G'/P';\nu_-)}
=1
\]
by Theorem \ref{thm:ps} (2).  
Thus Theorem \ref{thm:sbon} is proved.  
\end{proof}

\section{Concluding remarks}
\label{sec:conc}
We raise the following two related questions:
\begin{problem}
\label{prob:AB}
{\rm{
Find a condition
 on a pair
 of real reductive linear Lie groups
 $G \supset G'$
 such that the following properties (A)
 and (B)
 are satisfied.    
\begin{enumerate}
\item[{\rm{(A)}}]
All Shintani functions have moderate growth
 (Definition \ref{def:mod}), 
 namely,
 $\operatorname{Sh}_{\operatorname{mod}}(\lambda,\nu)
 =\operatorname{Sh} (\lambda,\nu)$
 for all $(\ZG G, \ZG {G'})$-infinitesimal characters
 $(\lambda,\nu)$.  
\item[{\rm{(B)}}]
The natural injective homomorphism
\begin{equation}
\label{eqn:HCrest}
\invHom{G'}{\pi^{\infty}|_{G'}}{\tau^{\infty}}
\hookrightarrow
\invHom{{\mathfrak {g}}',K'}{\pi_K}{\tau_{K'}}
\end{equation}
is bijective
 for any admissible smooth representations
 $\pi^{\infty}$ and $\tau^{\infty}$
 of $G$ and $G'$, 
respectively.  
\end{enumerate}
}}
\end{problem}
\begin{remark}
\label{rem:conj}
{\rm{
\begin{enumerate}
\item[1)]
If $G'=\{ e \}$
 then neither (A) nor (B) holds.  
\item[2)]
If $G=G'$ 
then (A) holds by the theory of
 asymptotic behaviors
 of Harish-Chandra's zonal spherical functions
 and (B) holds by the Casselman--Wallach 
theory 
 of the Fr{\'e}chet globalization
 \cite[Chapter 11]{WaI}.  
\item[3)]
If $G'=K$ then both (A) and (B) hold.  
\item[4)]
It is plausible
 that if $(G,G')$ satisfies 
 the geometric condition (PP)
 (Definition \ref{def:pp})
 then both (A) and (B) hold.  
\end{enumerate}
}}
\end{remark}
By using the argument in Sections \ref{sec:break} and \ref{sec:PS}
 on the construction of Shintani functions from 
symmetry breaking operators,
 we have the following:
\begin{proposition}
%\label{prop:}
For a pair of real reductive classical Lie groups
 $G \supset G'$, 
 {\rm{(B)}} implies {\rm{(A)}}.  
\end{proposition}
\begin{proof}
Let $\lambda_+$ and $\nu_-$ be as in Theorem \ref{thm:ps}.  
We denote by $\pi^{\infty}$ 
 the admissible smooth representation 
of $G$ on $C^{\infty}(G/P;\lambda_+)$
 and $\tau^{\infty}$ 
 the admissible smooth representation 
of $G'$ on $C^{\infty}(G'/P';\nu_-)$.  
Then 
 by Theorem \ref{thm:ps} (2), 
 we have the following linear isomorphism:
\[
\invHom{G'}{\pi^{\infty}|_{G'}}{\tau^{\infty}} 
\overset \sim \to 
\operatorname{Sh}_{\operatorname{mod}}(\lambda,\nu).  
\]
Similarly to the proof of  
 Theorem \ref{thm:ps} (2), 
 we have the natural bijection
\[
\invHom{G'}{\pi^{\omega}|_{G'}}{\tau^{\omega}} 
\simeq
\operatorname{Sh}(\lambda,\nu), 
\]
where $\pi^{\omega}$ is a continuous representation of $G$
 on the space of real analytic vectors
 of $\pi^{\infty}$, 
 and $\tau^{\omega}$ that of $\tau^{\infty}$.

In view of the canonical injective homomorphisms
\[
  \invHom{G'}{\pi^{\infty}|_{G'}}{\tau^{\infty}}
  \hookrightarrow
  \invHom{G'}{\pi^{\omega}|_{G'}}{\tau^{\omega}}
  \hookrightarrow
  \invHom{{\mathfrak {g}}', K'}{\pi_K}{\tau_K}, 
\]
 we see that if (B) holds, 
 then the inclusion 
\[
  \operatorname{Sh}_{\operatorname{mod}}(\lambda,\nu)
  \hookrightarrow
   \operatorname{Sh}(\lambda,\nu)
\]
is bijective.  
Hence the implication (B) $\Rightarrow$ (A)
 is proved.  
\end{proof}

\vskip 1pc
\subsection*{Acknowledgments}
Parts of the results and the idea of the proof were delivered
 in various occasions including 
 Summer School on Number Theory in Nagano (Japan)
 in 1995 organized by F. Sato,
 Distinguished Sackler Lectures
 organized by J. Bernstein at Tel Aviv University (Israel)
 in May 2007,
 the conference in honor of E. B. Vinberg's 70th birthday
 in Bielefeld (Germany) in July 2007
 organized by H. Abels, V. Chernousov, G. Margulis, D. Poguntke,
 and K. Tent, 
 \lq\lq{Lie Groups, Lie Algebras and their Representations}\rq\rq\
 in November 2011 in Berkeley (USA) organized by J. Wolf, 
  \lq\lq{Group Actions
 with Applications
 in Geometry and Analysis}\rq\rq\ 
 at Reims University
 (France) in June, 2013
 organized by J. Faraut,
 J. Hilgert,
 B. {\O}rsted, 
 M. Pevzner, 
 and B. Speh, 
 and 
 \lq\lq{Representations of Reductive Groups}\rq\rq\
 at University of Utah (USA) in July, 2013
 organized by J.~Adams, P.~Trapa, and D.~Vogan.  
The author is grateful for their warm hospitality.

The author was  partially supported
 by Grant-in-Aid for Scientific Research
 (B) (22340026)  and (A)(25247006).

\end{document}